\documentclass[10pt,a4paper,leqno]{amsart}
\usepackage{etex}
\usepackage{amsfonts}
\usepackage{mathrsfs}

\usepackage{comment}

\usepackage{amsthm}
\usepackage{xspace}
\usepackage{graphics}
\usepackage[english]{babel}
\usepackage{amsfonts,amssymb}
\usepackage[all]{xy}
\usepackage{stmaryrd}
\usepackage{color}

\usepackage{geometry}                
\usepackage{amssymb}
\usepackage{tikz}
\usetikzlibrary{calc}
\usetikzlibrary{decorations.pathreplacing,calc,decorations.pathmorphing}
\usepackage{tikz-qtree}
\usetikzlibrary{fit,shapes,calc,quotes,angles,arrows.meta,intersections} 
\usepackage{tkz-euclide}
\usetkzobj{all}
\newcommand{\gettikzxy}[3]{%
  \tikz@scan@one@point\pgfutil@firstofone#1\relax
  \edef#2{\the\pgf@x}%
  \edef#3{\the\pgf@y}%

}

\usepackage[colorlinks,linkcolor=blue,citecolor=blue,urlcolor=blue]{hyperref}

\def\clr{   \color{red}}

\def\clb{   \color{black}}

\newcommand{\tikzAngleOfLine}{\tikz@AngleOfLine}
  \def\tikz@AngleOfLine(#1)(#2)#3{%
  \pgfmathanglebetweenpoints{%
    \pgfpointanchor{#1}{center}}{%
    \pgfpointanchor{#2}{center}}
  \pgfmathsetmacro{#3}{\pgfmathresult}%
  }

\newcommand{\be}{\begin{otherlanguage}{english}}
\newcommand{\ee}{\end{otherlanguage}}

\theoremstyle{definition}
\newtheorem{defn}{Definition}
\newtheorem{que}{Question}
\newtheorem{rem}{Remark}
\theoremstyle{plain}
\newtheorem{lem}{Lemma}
\newtheorem{prop}{Proposition}
\newtheorem{thm}{Theorem}
\newtheorem{cor}[prop]{Corollary}

\newtheorem*{thm*}{Theorem}
\newtheorem*{con}{Conjecture}

\theoremstyle{remark}

\numberwithin{equation}{subsection}

\newcommand{\beq}{\begin{equation}}
\newcommand{\eeq}{\end{equation}}

\newcommand{\N}{\mathbb{N}}
\newcommand{\Z}{\mathbb{Z}}

\newcommand{\Q}{\mathbb{Q}}

\newcommand{\cC}{\mathcal {C}}
\newcommand{\cF}{\mathcal {F}}

\DeclareMathOperator{\id}{\mathrm Id}

\def\Id{{\rm Id}}

\def\id{{\rm Id}}

\newcommand{\norm}[1]{\| #1\|}

\def\cC{          \mathcal C}
\def\cD{          \mathcal D}
\def\cF{          \mathcal F}

\def\clr{   \color{red}}

\def\clb{   \color{black}}

\let\cal\mathcal

\def \R{{\mathbb R}}

\def \Z{{\mathbb Z}}
\def \N{{\mathbb N}}

\def\id{{\rm id}}

\newcommand{\T}{{\mathbb T}}
\newcommand{\prf}{{\begin{proof}}}
\newcommand{\epf}{{\end{proof}}}

\newcommand{\ary}{\begin{eqnarray}}
\newcommand{\eary}{\end{eqnarray}}

\newcommand{\aryst}{\begin{eqnarray*}}
\newcommand{\earyst}{\end{eqnarray*}}

\newcommand{\enmt}{\begin{enumerate}}
\newcommand{\eenmt}{\end{enumerate}}

\theoremstyle{definition}

\def\bee{\begin{equation}}
\def\eee{\end{equation}}

\theoremstyle{rema}

\newcommand{\Diff}{\text{Diff}}

\newcommand{\diff}{\rm Diff}

\numberwithin{equation}{section}

\begin{document}
\title[rigid pseudo-rotations on $\T^2$]{The rigidity of pseudo-rotations on the two-torus and a question of Norton-Sullivan}

\author{Jian Wang}
\address{Chern Institute of Mathematics, Nankai
University, Tianjin 300071,
P.R.China \newline\indent Max Planck Institute for Mathematics in
the Sciences, Inselstra{\ss}e 22, D-04103 Leipzig,
Germany}\email{wangjian@nankai.edu.cn}
\thanks{Jian Wang acknowledges the support of 
International Postdoctoral Exchange Fellowship Program [20130045] and National Natural Science Foundation of China [11401320].}

\author{Zhiyuan Zhang}
\address{Institut de Math\'{e}matique de Jussieu---Paris Rive Gauche, 
75205 Paris Cedex 13, FRANCE}
\email{zzzhangzhiyuan@gmail.com}


\date{June 23, 2017}
\maketitle

\begin{abstract}
We show that under certain boundedness condition, a $C^{r}$  conservative irrational pseudo-rotations on $\mathbb{T}^2$ with a generic rotation vector  is $C^{r-1}$-rigid. We also obtain $C^0$-rigidity for H\"older pseudo-rotations with similar properties. These provide a partial generalisation of the main results in \cite{Br,A}.

We then use these results to study conservative irrational pseudo-rotations on $\T^2$ with a generic rotation vector that is semi-conjugate to a translation via a semi-conjugacy homotopic to the identity. We show that the conservative centralizers of any such diffeomorphism is isomorphic to a uncountable subgroup of $\R^2/\Z^2$.  In connection with a question of Alec Norton and Dennis Sullivan, we describe the topologically linearizable maps within this class using the topology of the conservative centralizer group. In the minimal case, we obtain a precise characterization of topological linearizability for all totally irrational vectors.

We also construct a $C^\infty$ conservative and minimal  totally irrational pseudo-rotation diffeomorphism that is semi-conjugate to a translation, but is topologically nonlinearizable. This gives a negative answer to the question of Norton and Sullivan in the $C^{\infty}$ category.
\end{abstract}


\section{Introduction}
The question of linearization is one of the recurrent themes in dynamical systems, topology and analysis. As one of the earlier results, H. Poincar\'{e} proved the following celebrated classification of circle homeomorphisms:  a circle homeomorphism $f$ is semi-conjugate to an irrational  rigid rotation if and only if  the rotation number of $f$, denoted by $\rho(f)$, is irrational, which is equivalent to say that $f$ has no periodic orbits. 
Later A. Denjoy proved that $f$ is topologically conjugate to an irrational rigid rotation if it is a $C^1$ diffeomorphism of $\T^1$ without periodic points and $Df$ has bounded variation ($f\in C^{1+b.v.}$) \cite{D}. 
The linearization problem for circle diffeomorphisms with higher regularities were studied in great depth by M. Herman and J.-C. Yoccoz \cite{H, Yo95}.
 In the other direction, Denjoy (even before him, P. Bohl \cite{Bo}), provided examples of $C^1$ diffeomorphisms semi-conjugate but not topologically conjugate to an irrational rotation. Their examples were later improved to $C^{1+\alpha}$ for any $\alpha\in(0,1)$  by Herman  \cite{H}. 

It is attempting to generalise Poincar\'e's classification to homeomorphisms on higher dimensional manifolds. However, many new obstructions, such as the existence of mixing smooth diffeomorphisms with no periodic points (for example, \cite{Fayad}), prevented a simple statement as Poincar\'e's in the higher dimensions. 
Moreover, it is also a non-trivial task to extend the concept of rotation number to the study of higher dimensional dynamics. 
 Generalisations in this direction, in different forms, were introduced in \cite{SS, F}. In rough terms, we use \textit{rotation vectors}  to describe the asymptotic motion of orbits in the homology classes.  
Unlike the case of circle homeomorphisms, one usually obtain for more general dynamics a set of rotation vectors, which we call the \textit{rotation set}.
Under a condition called \textit{bounded mean motion}, T. J\"ager \cite{J} obtained an analoguous classification as Poincar\'e's for conservative pseudo-rotations on the 2-torus, that is, conservative homeomorphisms of $\T^2$ with rotation set reduced to a single vector.

 In this paper, we  study homeomorphisms of the two-dimensional torus $\T^2 = \R^2 / \Z^2$ which are isotopic to the identity. In this case, the rotation vectors and the rotation set are defined as follows.

Let $\mathrm{Homeo}_*(\mathbb{T}^2)$ be the group of homeomorphisms of $\mathbb{T}^2$ which are isotopic to $\mathrm{Id}_{\mathbb{T}^2}$. Any $f \in \mathrm{Homeo}_*(\mathbb{T}^2)$ admits a lift to $\R^2$, denoted by $\tilde{f}$, which is a homeomorphism of $\R^2$ satisfying $\pi\tilde{f} = f\pi$, where  $\pi:\R^2\rightarrow \T^2$ is the covering projection, i.e. $\pi(v) = v \mod \Z^2,\, \forall v \in \R^2$. 

For any lift of $f$, denoted by $\tilde{f}$, the \textit{pointwise rotation set} of $\tilde{f}$ is defined by 
\aryst
\rho_p(\tilde{f}) = \left\{ \rho(\tilde{f},z)  \mid   z \in \R^2, \quad \rho(\tilde{f}, z) = \lim_{n \to \infty} (\tilde{f}^{n}(z) - z)/n \mbox{ exists }\right \}. 
\earyst
M. Misiurewicz and K. Ziemian \cite{MZ} introduced the now standard definition of \textit{(Misiurewicz-Ziemian) rotation set}, which admits better properties:
\aryst
\rho(\tilde{f}) = \left\{ {\bf v} \in \R^2  \mid \frac{\tilde{f}^{n_i}(z_i)-z_i}{n_i} \to {\bf v},\mbox{for some } \{z_i\} \mbox{ in }
\mathbb{R}^2,\mbox{ and } \{n_i\} \mbox{ in }\N \mbox{ with }n_i \rightarrow \infty\right \}.
\earyst
The effect of changing the lift $\tilde{f}$ of $f$  is to translate $\rho_{p}(\tilde{f})$, $\rho(\tilde{f})$ by an integer vector.
In \cite{MZ}, the authors
 proved  that the rotation set $\rho(\tilde{f})$ is a compact convex subset of $\R^2$, giving rise to a basic trichotomy: $\rho(\tilde{f})$ is either a compact convex set with nonempty interior, a line segment, or a singleton. 
 
 There are many interesting problems and results  on the relation between the rotation set and the dynamics (see, e.g., \cite{MZ, FM,LM, LT,K16,KPS}). 
It is shown in \cite{LM, F2} that  a torus homeomorphism which is isotropic the identity, and has a rotation set with nonempty interior must have positive topological entropy. On the other hand, when the rotation set has empty interior, one seeks to obtain, to certain degree, a classification. In this direction, we have recent works \cite{J, JP, JT, K16, KPS}. A recent counter example of A. Avila to Franks-Misiurewicz's conjecture reveals some hidden complexity of this problem.
On the  extreme where  $\rho(\tilde{f})$ is a singleton, we say
that $f$ is a \emph{pseudo-rotation}.  In this case, it is clear that $\rho(\tilde{f},z)=\rho(\tilde{f})$ for every $z\in \mathbb{R}^2$, and for the convenience of the discussion, we write the set $\rho(\tilde{f})$ (resp. $\rho(f):=\rho(\tilde{f}) \mod \Z^2$) as a vector $\omega\in\R^2$ (resp. $\overline{\omega} \in\R^2/\Z^2$) instead of  $\{\omega\}$ (resp. $\{\overline{\omega}\}$).  J. Franks \cite{F3} showed that for a conservative pseudo-rotation, $\rho(f)$ is irrational (see Definition \ref{irrational vector}) if and only if $f$ has no periodic points.  We say that a pseudo-rotation $f$ has {\em bounded mean motion} (see Section \ref{sec: BDC and BMM}) if the {\em deviation} from the constant rotation $\tilde{f}^n(z)-z-n\rho(\tilde{f})$
are uniformly bounded in $z\in\R^2$ and $n\in \N$. 
This notion played an important role in the description of the dynamics, see, e.g., \cite{J,J1,LT, K16, SFGP, JS, BJ}. It is direct to see that
\begin{quote}
{\it $f$ is regularly\,\footnote{\,i.e. the semi-conjugacy map is homotopic to the identity (see \cite{J}).} semi-conjugate to an  irrational translation $\implies$ $f$ has bounded mean motion.}
 \end{quote}
 Conversely,
In \cite{J}, the author has shown that for an area-preserving totally irrational pseudo-rotation $f$ on $\T^2$,
\begin{quote}
{\it $f$ has bounded mean motion $\implies$ $f$ is semi-conjugate to an  irrational translation.}
\end{quote}

\smallskip

One of the purpose of this paper is to show that pseudo-rotations with bounded mean motion behave remarkably similar to that of rigid translations under iterations, in the following sense.

We say that a  $C^r$-diffeomorphism $f$ of a smooth manifold $M$ is {\em $C^k$-rigid}, where $0 \leq k\leq r \leq \infty$, if there exists a sequence $\{n_j\}_{j\geq0}$ in $\mathbb{N}$  such that $f^{n_j}\rightarrow \mathrm{Id}_{M}$ in the $C^k$-topology. When $M$ is the circle, the two-disc, or the two-torus, it is obvious that  $f$ is $C^0$-rigid if $f$ is topologically conjugate to a rigid rotation/translation. 

On any compact manifold $M$, a $C^0$-rigid (resp. $C^r$-rigid) non-periodic homeomorphism (resp.  $C^r$ diffeomorphism) has uncountably many commutators in ${\rm Homeo}(M)$ (resp. $\diff^{r}(M)$) (see \cite[Chap XII, (3.2)]{H}). Thus by a result of C. Bonatti, S. Crovisier and A. Wilkinson \cite{BCW}, the set of $C^1$-rigid diffeomorphisms is meagre in $\diff^{1}(M)$ for any compact manifold $M$. Their result, along with \cite{BCVW}, answered a question of S. Smale \cite{Smale} in the $C^1$-topology. On the other hand, rigid diffeomorphisms can be quite common under additional assumptions, for example, among pseudo-rotations. In this direction, we have Birkhoff's sphere conjecture (see Section \ref{subsop} for its statement), and the following recent results.

On the 2-disc $\mathbb{D}$, a pseudo-rotation is an area-preserving homeomorphism of $\mathbb{D}$ that fixes the origin, and has no other periodic point. Given  a pseudo-rotation $f$ on $\mathbb{D}$, there exists a lift $\tilde{f}$ of $f|_{\mathbb{D} \setminus \{(0,0)\}}$ to $\tilde{\mathbb{A}}=\mathbb{R}\times\,(0,1]$.  The rotation number $\rho(f)$ is defined as $\rho(\tilde{f}) \mod 1$ where $\rho(\tilde{f}) = \lim_{n \to \infty}(p_1(\tilde{f}^n(z)) - p_1(z))/n$. Here $p_1$ is projection of $\tilde{\mathbb{A}}$ to the $\R$-coordinate, and for a pseudo-rotation $f$, the limit is independent of $z \in \tilde{\mathbb{A}}$.

B. Bramham  \cite{Br} has shown that smooth 
pseudo-rotations of the disc with sufficiently Liouvillean rotation number are
$C^0$-rigid. 
 Bramham's proof uses the pseudo-holomorphic curve techniques from symplectic geometry. In 2015, Avila, B. Fayad, P. Le
Calvez, D. Xu and Z. Zhang  proved in \cite{A}  that a
pseudo-rotation of the disc of class $C^r ~(r\geq 2$)
is $C^{r-1}$-rigid if its rotation number is non-Brjuno (see Section \ref{sec pre} for the definition), which generalises the result of Bramham.\smallskip

 In this paper, we generalise some results of \cite{Br,A} to $\T^2$ under certain boundedness condition.  To properly state our result in its most general form, we introduce the following definitions. 

\begin{defn}\label{irrational vector}
For any integer $k \geq 2$, a vector $\omega\in \mathbb{R}^k$ is called  \emph{irrational} if $\omega\not\in \mathbb{Q}^k$. A vector $\omega=(\omega_1,\cdots,\omega_k)\in\mathbb{R}^k$ is called \emph{totally irrational} if $\omega_1,\cdots,\omega_k$ and 1 are rationally independent, i.e., the solution to the equation $\sum_{i=1}^kd_i\omega_i+e=0$ for $(d_1,\cdots,d_k,e)\in\mathbb{Z}^{k+1}$ contains  only $(0,\cdots,0)$. We say that a vector $\omega \in \R^2$ is {\em semi-irrational} if it is irrational but not totally irrational. We say that $\overline{\omega} \in \R^k/\Z^k$ is irrational (resp. totally irrational, semi-irrational for $k=2$) if there exists $\omega \in \R^k$ with $\omega \mod \Z^k= \overline{\omega}$, such that $\omega$ is irrational (resp. totally irrational, semi-irrational for $k=2$).
\end{defn}
We note that, for any semi-irrational $\omega=(\omega_1,\omega_2) \in \R^2$, there exists $(c,d)\in \Z^2\setminus\{(0,0)\}$ such that $c\omega_1+d\omega_2\in \Q$ with $gcd(c,d) =1$. We will see in Section \ref{sec:semi-irrational} that the set  $\{\pm(c,d)\}$ is uniquely determined by $\pi(\omega)$. We say that $\pm(d,-c)$ are the \emph{character vectors of $\omega$}.

 For any vector $\omega\in \mathbb{R}^2$, we set $\|\omega\|_{\T^2}=\min_{z\in \mathbb{Z}^2}\|\omega-z\|$.  Our first result states as follows.

\begin{thm}\label{sup}
Suppose that $f$ is an area-preserving pseudo-rotation of the torus that is H\"{o}lder with exponent $a\in(0,1]$ and $\rho(f)= \omega \mod \mathbb{Z}^2$ is irrational satisfying the following strong super-Liouvillean condition:
\begin{equation}\label{eq:a-super-liouvullen}\liminf_{n\rightarrow+\infty} n^{-1}a^{n}\ln \norm{n\omega}_{\T^2} = -\infty,\end{equation}
and one of the following conditions:
\begin{enumerate}
\item $f$ has bounded mean motion; 
\item $\omega$ is semi-irrational and $f^{\ell}$ satisfies  the bounded deviation parallel to a character vector of $\omega$ for some integer $\ell \geq 1$,
\end{enumerate}
then $f$ is $C^{0}$-rigid, i.e. $\liminf_{n \to \infty} d_{C^0}(f^{n},  \mathrm{Id}_{\mathbb{T}^2}) = 0$. 
\end{thm}

The condition (2) means that there exists a lift $\tilde{g}$ of $f^{\ell}$ such that the projection of the deviation $\tilde{g}^n(z)-z-n\rho(\tilde{g})$ to the orthogonal direction of the character vector of $\omega$ is uniformly bounded  for every $z\in\R^2$ and $n\in\N$. We defer to Section \ref{sec: BDC and BMM} for more details. When $\omega$ is a semi-irrational vector, we note that condition (2)  is  weaker than (1).

We also have the following result on the rigidity in higher topology, which is analogous to \cite[Theorem 1]{A}. We prove that for $r \geq 2$, $f$ is $C^{r-1}$-rigid if $f$ is a $C^r$ area-preserving  pseudo-rotation of $\mathbb{T}^2$ with super-Liouvillean rotation vector, satisfying the same boundedness condition as in Theorem \ref{sup}.  

\begin{thm}\label{nonBrjunont}
If $f$ is a $C^r$, $r\in \N_{\geq 2}$ (resp. $r=\infty$) area-preserving pseudo-rotation of $\mathbb{T}^2$ with rotation number $\rho(f)=\omega\mod \mathbb{Z}^2$. If one of the following cases is satisfied
\begin{enumerate}
\item $\omega$ is an irrational vector satisfying the following super-Liouvillean condition
\begin{equation}\label{eq:super-liouvullen}\liminf_{n\rightarrow+\infty} n^{-1}\ln \norm{n\omega}_{\T^2} = -\infty,\end{equation}

and $f$ has bounded mean motion; 
\item $\omega$ is an semi-irrational vector of strong non-Brjuno type, and $f^{\ell}$ satisfies  the bounded deviation parallel to the character vector of $\omega$ for some integer $\ell \geq 1$,
\end{enumerate}
then $f$ is $C^{r-1}$-rigid (resp. $C^{\infty}$-rigid).
\end{thm}

The notion of strong non-Brjuno type will be given in Definition \ref{semi-irrational} after some preparations in Section \ref{sec:semi-irrational}. In loose terms, we can always normalise a semi-irrational vector $\omega$ to obtain a single irrational number which we call the \textit{character frequency} of $\omega$, and in Theorem \ref{nonBrjunont} (2), we study those $\omega$ with character frequency which is non-Brjuno in the classical sense. When $\omega$ is semi-irrational, the strong non-Brjuno condition is much weaker than the super-Liouvillean condition \eqref{eq:super-liouvullen}.

Our proofs of Theorem \ref{sup} and  \ref{nonBrjunont} are, to a large extent, based on the strategy in \cite{A}.
The main tool in our proof is Proposition \ref{lem:displaced disk b in bdc} which is a generalisation of  \cite[Lemma 3.1]{A} to $\T^2$. It gives a fine control of the $C^0$ displacement of a pseudo-rotation on $\T^2$ with certain boundedness condition using the modulus of the rotation vector. This result also plays a pivotal role in the proof of Theorem \ref{cortoplinear2} below. 
In Section \ref{subsop}, we introduce some further perspectives of  this proposition. In particular, we hope it would be useful in producing progress toward Question \ref{queakontorus}.
\begin{rem}
By a standard argument (see, e.g., \cite[Appendix A.2]{Br}), we can show that the set of totally irrational, super-Liouvillean vectors   is $G_{\delta}$ dense in $\mathbb{R}^2$, that is, topologically generic.  
\end{rem}
Our next result is motivated by a line of research on the extension of the Denjoy's type example on the circle to $\T^2$. 
One motivating question is the \textit{wandering domains problem} (see \cite{NS}): 
Can one ``\,blow up\,'' one or more orbits of $T_{\alpha}$ to make a smooth diffeomorphism with wandering domains?  
We say that a homeomorphism of $\T^2$ obtained by blowing-up finitely many orbits of an irrational translation is of Denjoy type.
By the classical KAM theory,  any $C^{\infty}$ volume-preserving pseudo-rotation of $\T^n$ with Diophantine rotation vector $\alpha \in \T^n$, which is sufficiently close to $T_{\alpha}$, is smoothly conjugate to $T_{\alpha}$. 
P. McSwiggen in \cite{McS} constructed a $C^{2+\alpha}$ diffeomorphism of Denjoy type having a smooth wandering domain. In particular, his example is not topologically conjugate to a rigid translation.
A. Norton and D. Sullivan
 in \cite{NS} showed that no $C^3$ diffeomorphism on $\T^2$ of Denjoy type exists with circular wandering domains, and asked the following question:
\begin{que}[Norton and Sullivan, 1996]\label{qnortonsullivan} 
If $f:\mathbb{T}^2\rightarrow\mathbb{T}^2$ is a diffeomorphism, $h:\T^2\rightarrow\T^2$ is a continuous map homotopic to the identity, and $h f=T_{\rho} h$ where $\rho\in\R^2$ is a totally irrational vector, are there natural geometric conditions (e.g. smoothness) on $f$ that force $h$ to be a homeomorphism?
\end{que}
\noindent In \cite{PS}, A. Passeggi and M. Sambarino  also mentioned the question that whether there exists $r$ so that if $f : \T^2 \to \T^2$ is a $C^r$ diffeomorphism semi-conjugate to an ergodic translation then $f$ is conjugate to it.  For more recent developments, we mention \cite{Kwa, KM, Kara, Nav}. For a survey  on related problems, see  S. van Strien \cite{vS}.

In the next result, we formalise a natural geometric condition, using the concept of centralizers, which implies topological linearizability.
For any $\overline{\omega} \in \R^2/\Z^2$, we denote by $\cal{C}_{\overline{\omega}}$ the set of maps in ${\rm Homeo}_{*}(\T^2)$ which are regularly semi-conjugate to the translation $T_{\overline{\omega}}$, i.e. $f \in \cal{C}_{\overline{\omega}} \iff$ there exists a surjective continuous map $h : \T^2 \to \T^2$ homotopic to ${\rm Id_{\T^2}}$ such that $hf = T_{\overline{\omega}}h$. A homeomorphism $g\in {\rm Homeo}(\T^2)$ is called a centralizer of $f$ if $gf=fg$.

\begin{thm}\label{cortoplinear2}
Let $r \in \N_{\geq 2}$ or $r=\infty$, and let $f$ be a  map in $\diff^r(\T^{2}, {\rm Vol}) \cap \cal{C}_{\overline{\omega}}$ with $\overline{\omega}$ super-Liouvillean \eqref{eq:super-liouvullen} totally irrational, and let $G_{f}$ be the set of centralizers of $f$ in $\diff^{r-1}(\T^2, {\rm Vol}) \cap \mathrm{Homeo}_*(\mathbb{T}^2)$. Then $G_{f}$ is isomorphic to a uncountable subgroup of $\R^2/\Z^2$. Moreover, we have the following implications:
\begin{quote}
$G_{f}$ is compact in the $C^0$ topology$\implies$ $f$ is topologically linearizable$\implies$$G_{f}$ is pre-compact in the $C^0$ topology.
\end{quote}
\end{thm}
In Theorem \ref{cortoplinear2}, we obtain information on the group structure of the centralizers (e.g. abelian), and establish a close link between the topological linearizability with the topology of the centralizers  for a conservative map in $\cal{\cC}_{\overline{\omega}}$ with a generic rotation vector $\overline{\omega}$. In the minimal case, we have the following precise characterization, for any totally irrational $\overline{\omega}$.

\begin{thm}\label{cortoplinearminimal}
Let $f$ be a minimal map in ${\rm Homeo}(\T^2, {\rm Vol}) \cap \cal{C}_{\overline{\omega}}$ with $\overline{\omega}$ totally irrational, and let $G^{0}_{f}$ be the set of centralizers of $f$ in ${\rm Homeo}(\T^2, {\rm Vol}) \cap \mathrm{Homeo}_*(\mathbb{T}^2)$. Then $G^{0}_{f}$ is isomorphic to a subgroup of $\R^2/\Z^2$. Moreover, 
$G^{0}_{f}$ is compact in the $C^0$ topology$\iff$$f$ is topologically linearizable.
\end{thm}
We do not know whether a $C^2$ conservative map in $\cal{C}_{\overline{\omega}}$ for a totally irrational $\overline{\omega}$ is always minimal. In the non-conservative setting, for a transitive and non-minimal example, see \cite[Theorem 1.2]{BCJR}.  Recently, by adapting the proof of Theorem \ref{cortoplinear2} and Proposition \ref{lem:displaced disk b in bdc}, we have shown that: for a totally irrational pseudo-rotation $f$ of $\T^2$ (not necessarily area preserving), $f$ is topologically linearizable if and only if $\{f^n\}_{n \in \Z}$ is pre-compact in the $C^0$ topology. We will treat this in a separate note. 

For the centralizers of smooth circle diffeomorphisms, a deep study was done  by Herman and Yoccoz \cite{H, Yo95}.
Herman \cite[Chap XII]{H} has constructed uncountable centralizers for smooth diffeomorphisms of the circle with irrational rotation numbers which are not smoothly linearizable, giving a counterexample to a question of H. Rosenberg and W. Thurston \cite{RT}: let $\cal{F}$ be a foliation of $\T^3$ with all leaves planes $\R^2$, is $\cF$ differentiably conjugate to a linear foliation? R.  P\'erez-Marco \cite{PM} later constructed analytic circle diffeomorphisms, and germs of holomorphic diffeomorphisms of $(\mathbb{C}, 0)$ with similar properties.
Conversely, Yoccoz constructed a $C^{\infty}$ diffeomorphism of the circle, with irrational rotation number, and with centralizers reduced to its iterates in $\diff^2(\T)$. He also showed that generically the centralizers of a $C^{\infty}$ diffeomorphism of the circle with irrational rotation number is the limit set of the group of its iterates in the $C^{\infty}$ topology, but this does not hold without the genericity condition. Smooth nonlinearizable diffeomorphisms in higher dimension can be constructed using Anosov-Katok's method, e.g. \cite{AK, FS}. But most of the previous constructions are wild, e.g. weak-mixing,  and do not admit a semi-conjugacy (see \cite{JK17} for a case where both semi-conjugacy and topologically nonlinearizability are obtain, based on certain classification result they proved). 

We will show in Theorem \ref{thm example} that the set of maps in Theorem \ref{cortoplinear2} also includes topologically nonlinearizable maps, by producing a $f$ with $G_{f}$ that is not pre-compact in the $C^0$ topology.
We constructed a $C^\infty$ conservative (resp. minimal)  totally irrational pseudo-rotation $f$ with bounded mean motion that is not topologically conjugate to a translation. Our construction combine the classical Anosov-Katok method (see \cite{AK,FK}), with J\"ager's theorem (Theorem \ref{thmJager}). We thus give a negative answer to the above question of Norton and Sullivan in the $C^{\infty}$ category:  in general, not even the infinite smoothness condition can force $h$ to be a homeomorphism.

\begin{thm}\label{thm example}
For any integer $d \geq 2$, there exists a $C^{\infty}$ area-preserving and minimal diffeomorphism $f : \T^d \to \T^d$ which is semi-conjugate to a translation by a map homotopic to the identity, but is not topologically conjugate to a translation. Moreover, we can require $f$ to have super-Liouvillean rotation vector.
\end{thm}
\begin{rem}\label{nonrigidandboundedmeanmotion}
By Theorem \ref{thm example}, there exists $f$ satisfying the conditions of Theorem \ref{nonBrjunont} and \ref{cortoplinear2}, which is not topologically linearizable. From the proof of Theorem \ref{thm example}, we easily see that $\{f^n\}_{n \in \N}$ is not pre-compact in the $C^0$ topology. Thus the last item in Theorem \ref{cortoplinear2} is not a consequence of its condition.
\end{rem}
We mention that a similar result on $\mathbb{D}$ was obtain by \cite{JK17} using a different method. In particular, in their case, they can classify all the semi-conjugacies \cite[Corollary 1.3]{JK17}.

\subsection{{\sc Open problems}} \label{subsop}

As a natural by-product of the construction in Theorem \ref{thm example}, the pseudo-rotation $f$ we obtained can have  super-Liouvillean rotation vector satisfying \eqref{eq:super-liouvullen}, and bounded mean motion, in which case, it is $C^{\infty}$-rigid due to Theorem \ref{nonBrjunont}.\footnote{\,Instead of using Theorem \ref{nonBrjunont}, we can also require $f$ in Theorem \ref{thm example} to be $C^{\infty}$-rigid by a direct adaption of the construction. }
\clb
 Therefore, the following questions seem natural.
\begin{que}
Is a $C^r~(r\geq1 \text{ or } r=\infty)$ conservative  irrational pseudo-rotation with bounded mean motion always $C^0$-rigid? If yes, with which type of irrational vector  a $C^r$ conservative irrational pseudo-rotation   can be non-conjugate to a translation? Otherwise,  with which type of irrational vector  a $C^r$ conservative irrational pseudo-rotation  can be not $C^0$-rigid? \end{que}

Another natural question is following.
\begin{que}
Does Question \ref{qnortonsullivan} have a positive answer for analytic diffeomorphisms ?
\end{que}

Our next question is motivated by the desire to further understand Question \ref{qnortonsullivan}, in connection with  the following result,
recently announced by Avila and R. Krikorian:
\begin{quote}
There exists a neighbourhood $V$ (for the $C^{\infty}$ topology) of the set of rigid rotations on the disk $\mathbb{D}$ such that each pseudo-rotation $f$ in $V$ is almost-reducible.
\end{quote}
Here a $C^{\infty}$ diffeomorphism $f$ on $\mathbb{D}$ is said almost-reducible if there exists a sequences of $C^{\infty}$ area-preserving diffeomorphisms $h_n$ such that $h_n^{-1} f h_n$ converges to  $T_{\rho(f)}$ in the $C^{\infty}$ topology.
The motivation of their work goes back to the following old conjecture of Birkhoff
which is still unsolved (see \cite[Page 712]{Bi3} and
\cite{H}): \begin{con}[Birkhoff's sphere conjecture]Let $f$ be an orientation preserving, real-analytic, Lebesgue
measure-preserving diffeomorphism of the 2-sphere $\mathbb{S}^2$,
and having only two periodic (necessarily fixed) points. Then $f$
is conjugate to a rigid irrational rotation.\end{con}
An important ingredient \footnote{\,From the announcement of Artur Avila at the conference in memory of Jean-Christophe Yoccoz. } in Avila-Krikorian's proof is an {\it a priori bound} for a renormalisation scheme obtained by using \cite[Lemma 3.1]{A}. Since our proof of Theorem \ref{sup} and \ref{nonBrjunont} are also based on a generalisation of \cite[Lemma 3.1]{A} to $\T^2$ under certain boundedness condition, namely our Proposition \ref{lem:displaced disk b in bdc}, it is then natural to ask if the following weaker version of Norton-Sullivan's question could be true in the $C^{\infty}$ category.
\begin{que}\label{queakontorus}
Suppose that  $f \in \mathrm{Homeo}_*(\T^2) \cap \diff^{\infty}(\T^2)$ is a conservative totally irrational pseudo-rotation regularly semi-conjugate to a translation, then is $f$ almost-reducible?
\end{que}

By Remark \ref{nonrigidandboundedmeanmotion}, we can see that, even for a local result on Question \ref{queakontorus} in analogue to Avila-Krikorian's, the almost-reducibility cannot be replaced by topological conjugacy. Indeed, for any $f$ satisfying both Theorem \ref{nonBrjunont} and Theorem \ref{thm example}, a sequence of iterates of $f$ will accumulate at the identity. But $f^{\ell}$ is not topologically linearizable for any integer $\ell \geq 1$, for otherwise we would obtain a non-translation homeomorphism $hfh^{-1}$ which commutes with a minimal translation $T_{\rho(f)}$.

\smallskip
This article is organized as follows. In Section \ref{sec pre}, we introduce some notations, and recall some classical definitions and results on the plane. In Section \ref{sec:proof}, we prove Proposition \ref{lem:displaced disk b in bdc} and Corollary \ref{cor} which are the key step to prove Theorem \ref{sup} and  \ref{nonBrjunont}. We prove Theorem \ref{sup} and  \ref{nonBrjunont} in Section \ref{sec:thmsupnonbrju}. We then prove Theorem \ref{cortoplinear2} and \ref{cortoplinearminimal} in Section \ref{sec:thmcorto}. We prove Theorem \ref{thm example}  in Section \ref{secthm4}. \clb
\smallskip




\section{Preliminaries}\label{sec pre}
\subsection{Some properties of the rotation set}\label{sec rotationsets}\quad\par
For any $v \in \T^2$ or $\R^2$, let $T_{v} : \T^2 \to \T^2$ denote the map  given by $T_{v}(z) = z + v$.
For a given $f \in \mathrm{Homeo}_{*}(\T^2)$, let $\tilde{f}$ be a lift of $f$ to $\R^2$.
By the definition of $\rho(\tilde{f})$, we easily deduce the
following elementary properties:
\begin{enumerate}\label{prop:ROT}
  \item[1.] $\rho(T_{k} \tilde{f})=\rho(\tilde{f})+k$
  for every $k\in \mathbb{Z}^2$;
  \item[2.] $\rho(\tilde{f}^q)=q\rho(\tilde{f})$ for every $q\in \mathbb{N}$.
\end{enumerate}\smallskip

We recall that the group $\mathrm{SL}(2, \R)$ acts on $\R^2$ by  affine automorphisms: for any matrix $A = \begin{bmatrix} a & b \\ c & d \end{bmatrix} \in \mathrm{SL}(2, \R)$,  we set $A \cdot (x,y) = (ax+by, cx+dy)$. For any $A \in \mathrm{SL}(2,\Z)$, let $T_A : \T^2 \to \T^2$ denote the unique $C^{\infty}$ area-preserving  diffeomorphism  such that $ \pi A=T_A\pi$.

Given any $f \in \mathrm{Homeo}(\T^2)$ and  $A \in \mathrm{SL}(2,\Z)$, we set  $f_{A}=T_{A} f T_{A}^{-1}$.  Assume in addition that $f \in \mathrm{Homeo}_{*}(\T^2)$, and let $\tilde{f}$ be a lift of $f$ to $\R^2$, then $\tilde{f}_{A} := A\tilde{f}A^{-1}$ is a lift of $f_{A}\in \mathrm{Homeo}_*(\T^2)$, and $\rho(\tilde{f}^q_{A})=qA\cdot \rho(\tilde{f})$ for any $q\in\N$ (see, e.g. \cite[Section 0.3.3]{K07}). 
It is  direct to see that: if $f$ preserves the Lebesgue measure on $\T^2$, then so does $f_A$; and for any $r \in \N \cup \{\infty\}$,  $f$ is $C^r$-rigid if and only if $f_A$ is $C^r$-rigid. \clb

\subsection{Bounded  deviation condition and bounded mean motion}\label{sec: BDC and BMM} \quad

We denote by $\langle\, ,\rangle$, resp. $\|\cdot\|$, the standard scalar product, resp. the Euclidean norm on $\mathbb{R}^2$.
For any $v=(v_1,v_2)\in \mathbb{R}^2\setminus\{(0,0)\}$, we denote by $v^\bot$ the orthogonal unit vector $v^\bot=(-v_2,v_1)/\|v\|$. 

Let us recall the bounded  deviation condition and bounded mean motion property (see \cite{J}). 
 \begin{defn}\label{BDC} Let $f$ be a pseudo-rotation of $\T^2$. We say that  $f$ has \emph{bounded mean
motion } (with a bound $\kappa\geq0$) if there exists $\tilde{f}$, a lift of $f$,  such that   for any $ z\in
\mathbb{R}^2$ and $ n\in \mathbb{N}$,
\begin{equation}\label{eq:bmm}\|\tilde{f}^n(z)-z-n\rho(\tilde{f})\|\leq \kappa. \end{equation}  We say that  $f$ has  \emph{bounded deviation parallel to $v \in \R^2 \setminus \{(0,0)\}$} (with a bound $\kappa\geq0$), if there exists  $\tilde{f}$, a lift of $f$, such that  for any $z\in \mathbb{R}^2$ and $n\in\mathbb{N}$,
\begin{equation}\label{eq:bdc}
|\langle \tilde{f}^n(z)-z-n\rho(\tilde{f}), v^{\perp} \rangle|\leq \kappa.
\end{equation} 
We note that the terms on the left hand sides of \eqref{eq:bmm}, \eqref{eq:bdc} are independent of the choice of the lift $\tilde{f}$. It is also clear that \eqref{eq:bmm} implies \eqref{eq:bdc} for any $v \in \R^2 \setminus \{(0,0)\}$.
\end{defn}

Given any $\kappa \geq 0$, any $\overline{\omega}\in \R^2/\Z^2$, and any $v \in \R^2 \setminus \{(0,0)\}$, we let $\mathcal{C}_{\kappa,\overline{\omega}}$  (resp. $\mathcal{D}_{\kappa,\overline{\omega},v}$) be the set of  $f \in
\mathrm{Homeo}_*(\mathbb{T}^2)$ such that $\rho(f) = \overline{\omega}$ and there exists $\tilde{f}$,  a lift of $f$ to $\mathbb{R}^2$, satisfying 
bounded mean motion (\ref{eq:bmm}) (resp. bounded deviation condition (\ref{eq:bdc})) with a bound $\kappa$.

For any $m\in \mathbb{N}$ and $A \in \mathrm{SL}(2,\Z)$, it is direct to see that if $f \in \mathcal{C}_{\kappa,\overline{\omega}}$, then $f^m \in \mathcal{C}_{\kappa,m\overline{\omega}}$ and $f_A \in \mathcal{C}_{\|A\|\kappa, T_A (\overline{\omega}) }$. Indeed, for any $n\in\N$ and $z\in\R^2$ we have 
\ary
\label{property of ABMM}
&\|(A\tilde{f}A^{-1})^n(z)-z-n\rho(A\tilde{f}A^{-1})\|=\|A
\tilde{f}^{n}A^{-1}(z)-AA^{-1}(z)-nA\cdot\rho(\tilde{f})\|\leq \|A\|\kappa;  \\
\label{property of BMM}
&\|(\tilde{f}^m)^n(z)-z-n\rho(\tilde{f}^m)\|
=\|\tilde{f}^{nm}(z)-z-nm\rho(\tilde{f})\|\leq \kappa.
\eary
Similarly, we can directly verify that for any $m \in \N$, any $f \in \mathcal{D}_{\kappa, \overline{\omega},v}$, we have $f^{m} \in \mathcal{D}_{\kappa, m\overline{\omega},v}$.
We also have the following lemma:
\begin{lem}\label{lem:bdc} For any $A \in \mathrm{SL}(2,\Z)$, any $f \in \mathcal{D}_{\kappa, \overline{\omega},v}$, we have $f_A \in \mathcal{D}_{\|A\|\kappa,T_A(\overline{\omega}), A\cdot v}$. \end{lem}
\begin{proof}
Let $\tilde{f}$ be a lift of $f$ satisfying \eqref{eq:bdc}.
We have seen that $\rho(A\tilde{f} A^{-1}) = A \cdot \rho(\tilde{f})$. Moreover, we can directly verify that $(A\cdot v)^{\bot} = \frac{(A^T)^{-1}\cdot v^{\bot}}{\|(A^T)^{-1}\cdot v^{\bot}\|}$, where  $A^T$ is the transpose of $A$.
Thus for any  $n\in\mathbb{N}$ and $z\in \mathbb{R}^2$, we have 
\begin{eqnarray*}\label{eq:Abdc}
 & &|\langle A\tilde{f}^nA^{-1}(z)-z-n\cdot\rho(A\tilde{f}A^{-1}),(A\cdot v)^\bot\rangle|\nonumber
\\ &=& \left|\langle A \cdot (
\tilde{f}^{n}A^{-1}(z)-A^{-1}(z)-n\rho(\tilde{f})), \frac{(A^T)^{-1}\cdot v^\bot}{\|(A^T)^{-1}\cdot v^\bot\|}\rangle\right|\nonumber
\leq
 \|A\|\kappa.\nonumber
 \end{eqnarray*}
 \end{proof}

We will use the following result by  J\"ager, contained in \cite[Proposition A and Theorem C]{J}.
\begin{thm}\label{thmJager}
Suppose that $f \in \mathrm{Homeo}_*(\T^2)$ is a conservative (resp. minimal) totally irrational pseudo-rotation with bounded mean motion. Then $f$ is semi-conjugate to the irrational rotation on $\T^2$, i.e. there exists a continuous surjection $h : \T^2 \to \T^2$ such that $hf = T_{\omega}h$ for some  totally irrational $\omega \in \T^2$. Moreover, if $f$ is minimal, then one can take $h$ to be homotopic to $\Id_{\T^2}$ 
\end{thm}

\subsection{On semi-irrational vectors}\label{sec:semi-irrational}\quad

In this section, we define several quantities associated to a semi-irrational vector $\omega = (\omega_1, \omega_2) \in \R^2$. In the following, we set $\overline{\omega} = \omega \mod \Z^2 \in \R^2/\Z^2$.

By definition, there exist $(c,d)\in\Z^2\setminus\{(0,0)\}$,  and $e=p/q\in\Q$ with $p,q \in \Z$, $q > 0$,  such that $c\omega_1+d\omega_2+e=0$ and $gcd(c,d) = gcd(p,q)=1$. Here $gcd$ denotes the greatest common denominator. Since $\omega$ is semi-irrational, we deduce that the choice of $(c,d,e)$ is unique up to a sign. Hence the set $\{\pm(d,-c)\}$  and the integer $q$ are uniquely determined by $\omega$. Moreover, it is easy to check that $\{\pm(d,-c)\}$   and $q$ only depend on $\overline{\omega} = \omega \mod \Z^2$. We will call $\pm(d,-c)$  ( resp. integer $q$ )  the {\em character vectors } (resp. {\em character number}) of $\overline{\omega}$.
 \smallskip

Let $v = (d,-c)$ be a character vector of $\overline{\omega}$ defined as above. We choose $a,b\in\Z$ such that $A=\begin{bmatrix} a & b \\ c & d \end{bmatrix}  \in \mathrm{SL}(2,\mathbb{Z})$. The existence of $A$ follows from $gcd(c,d)=1$. We thus obtain $qA\cdot \omega+(0,p)=(\alpha,0)$, where $\alpha = q(a\omega_1 + b\omega_2)$. It is clear that: $\alpha\not\in\Q$ for otherwise $\omega \in \Q^2$; and $\alpha \mod \Z$ depends only on $q, A, v$ and $\overline{\omega}$.
We note that: 
\enmt
\item given a character vector $v=(d,-c)$, the constant $\alpha \mod \Z$ is independent of $A$. 
Indeed, for any $a',b' \in \Z$ such that $A' =  \begin{bmatrix} a' & b' \\ c & d \end{bmatrix}  \in \mathrm{SL}(2,\mathbb{Z})$, we have $(a',b') \in (a,b) + \Z(c,d)$. As a consequence, let $qA'\cdot \omega + (0,p) =: (\alpha',0)$, we have $\alpha' - \alpha \in p\Z$;
\item for any $a \in \R$, we let  $\norm{a}_{\T}$ denote the distance between $a$ and the closest integer. Then $\norm{\alpha}_{\T}$ is independent of the choice of $v$:  this follows from  replacing $(A,v)$ by $(-A,-v)$ in the above discussion, and observing that $\norm{\alpha}_{\T} = \norm{(-\alpha)}_{\T}$.
\eenmt
We will say that $\cal{F}(\overline{\omega}) := \norm{\alpha}_{\T}$ is the {\em character frequency} of $\overline{\omega}$. We have the following lemma.
\begin{lem}\label{lem:Abdc}
Let $f$ be a semi-irrational pseudo-rotation on $\T^2$ with $\rho(f) = \overline{\omega}$, and there exists an integer $\ell \geq 1$ such that $f^{\ell}$ has bounded deviation  parallel to a character vector of $\overline{\omega}$ with a bound $\kappa\geq0$. Then there exist an integer $L \geq 1$ and $A \in \mathrm{SL}(2,\Z)$, such that $f' :=T_{A} f^{L} T_{A}^{-1} \in  \mathcal{D}_{\|A\|\kappa, \pi(\ell \beta,0), (1,0)}$, where $\beta = \cal{F}(\overline{\omega})$. 
\end{lem}
\begin{proof} 
 We let $q,\beta$ be respectively  the character number and the character frequency of $\overline{\omega}$. We can choose a character vector $v$ of $\overline{\omega}$ such that, by setting $A, \alpha$ as above, we have $\alpha - \beta \in \Z$. By hypothesis,  $f^{\ell} \in \mathcal{\cD}_{\kappa, \ell \overline{\omega}, v}$, and as a result $f^{\ell q} \in \mathcal{\cD}_{\kappa, \ell  q \overline{\omega}, v}$. By Lemma \ref{lem:bdc}, we have that $f' :=T_{A} f^{\ell q} T_{A}^{-1} \in  \mathcal{\cD}_{\norm{A}\kappa,  \ell q T_A (\overline{\omega}), A \cdot v}$. Note that by the discussion above, $q T_A(\overline{\omega}) = (\alpha, 0) \mod \Z^2 = (\beta, 0) \mod \Z^2$ and $A \cdot v = (1,0) $. We conclude the proof by letting $L = \ell q$.
\end{proof}
\begin{rem} \label{ftof'}
For any $r \in \N \cup \{\infty\}$, if $f'$ in Lemma \ref{lem:Abdc} is $C^r$-rigid, then $f$ is also $C^r$-rigid. 

\end{rem}

\subsection{Rational approximations of an irrational vector}\label{draiv}\quad\par


In order to consider the class of rotation vectors in the semi-irrational case in Theorem \ref{nonBrjunont}, i.e. strong non-Brjuno type, we introduce the following definitions.

 We recall that the sequence of denominators of the best rational approximations of  $\alpha\in\mathbb{R}\setminus\mathbb{Q}$, denoted by $\{q_n = q_n(\alpha) \}_{n\geq 0}\subset\mathbb{N}$ satisfies that $q_0=1$, and for each $n \in \N$ that
\begin{equation}\label{DBRA}
(1)\,\,q_n<q_{n+1},\quad (2)\,\norm{q_n\alpha}_{\T} \leq \norm{q\alpha}_{\T} \quad \forall 1\leq q< q_{n+1},\quad\text{and}\quad (3)\,\, \norm{q_n\alpha}_{\T} < \frac{1}{q_{n+1}}. 
 \end{equation}
We recall that an irrational number  $\alpha$ is of {\em Brjuno type} (see \cite{Brj,Yo})   if $\sum_{n=0}^{+\infty}\frac{\ln q_{n+1}}{q_n} < \infty$.
We say that $\alpha$ is of non-Brjuno type if it is not of Brjuno type.

 When $\omega\in \R^2\setminus\Z^2$ is an irrational vector, there is an analogous definition of Brjuno vectors in \cite{GL}. In our case where $\omega$ is a semi-irrational vector, we introduce the following definition which is stronger than the one in \cite{GL}, but it is more natural in our case. 
\begin{defn}\label{semi-irrational} 
We say that a semi-irrational vector $\omega\in \R^2$ or $\overline{\omega} \in \R^2/\Z^2$ is of {\em strong non-Brjuno type}  if $\cal{F}(\overline{\omega})$, the character frequency of $\overline{\omega}$,  is of non-Brjuno type.
\end{defn}
For a semi-irrational vector in $\R^2$,  the super-Liouvillean condition \eqref{eq:super-liouvullen} implies the strong non-Brjuno condition.

\subsection{Franks' Lemma}\quad\par

A \emph{free disk chain} for a homeomorphism $\tilde{f}$ of $\mathbb{R}^2$
is a finite set $b_i\,\,(i=1,2,\cdots,n)$ of homeomorphically embedded open disks in
$\mathbb{R}^2$ satisfying
\begin{enumerate}
                            \item $\tilde{f}(b_i)\cap
                            b_i=\emptyset$ for $1\leq i\leq
                            n$;
                            \item if $i\neq j$ then either $b_i=b_j$
                            or $b_i\cap b_j=\emptyset$;
                            \item for $1\leq i\leq
                            n$, there exists $m_i>0$ such that $\tilde{f}^{m_i}(b_i)\cap
                            b_{i+1}\neq\emptyset$.
                          \end{enumerate}
 We say that $\{b_i\}_{i=1}^n$ is a \emph{periodic free
disk chain} if $b_1=b_n$.
\smallskip

In \cite{F1}, J. Franks proved the following useful lemma about the
existence of fixed points of an orientation preserving homeomorphism
$\tilde{f}$ of $\mathbb{R}^2$ from Brouwer theory.

\begin{prop}[Franks' Lemma]\label{prop:Franks' Lemma}
Let $\tilde{f}: \mathbb{R}^2\rightarrow \mathbb{R}^2$ be an orientation
preserving homeomorphism which possesses a periodic free disk chain.
Then $\tilde{f}$ has at least one fixed point.
\end{prop}

\section{Estimate on the displacement of a conservative pseudo-rotation of $\T^2$}\label{sec:proof}

In this section, we estimate  the maximal $C^0$ displacement of a conservative pseudo-rotation  of the two-torus with respect to its rotation vector when the homeomorphism satisfies certain boundedness condition.

Let $\T^2$ be endowed with the standard metric  induced by the Euclidean metric on $\mathbb{R}^2$, and denote by $\lambda$  the Lebesgue measure on $\mathbb{T}^2$. A measurable subset $F \subset \R^2$ is called a \textit{fundamental domain} under the action of $\Z^2$ if the union of $\{ F+v \mid v \in \Z^2 \}$ covers $\R^2$, and $F \cap (F+v) = \emptyset$ for any $v \in \Z^2 \setminus \{(0,0)\}$. A fundamental domain $F$ is bounded if ${\rm diam}(F) := \sup_{x,y \in F}\norm{x-y} < \infty$.
We say that a topological open disc $D$ of $\mathbb{T}^2$ is \emph{simple} with respect to a bounded fundamental domain  $F$, if there is a connected component  $\widetilde{D}$ of $\pi^{-1}(D)$ in $\mathbb{R}^2$ 
 such that $\widetilde{D}\subset F$. 
 
The key step in proving Theorem  \ref{sup} and Theorem \ref{nonBrjunont} is the following proposition:

\begin{prop}\label{lem:displaced disk b in bdc}Let $f$ be  an area-preserving pseudo-rotation of
$\mathbb{T}^2$ and $F \subset \R^2$ be a bounded fundamental domain. If there exist  $\kappa \geq0$ and $\tilde{f}$, a lift of $f$, with $\rho(\tilde{f})=\omega=(\omega_1,\omega_2) \in
\mathbb{R}^2\setminus\{(0,0)\}$, satisfying 
\eqref{eq:bdc} with $v= \omega$ (in particular, if $f$ has bounded mean motion with a bound $\kappa$), then any simple open disc $D$ with respect to $F$
such that $\lambda(D)>c(\kappa, F)\|\omega\|$ satisfies $f(D)\cap D\neq
\emptyset$, where $c(\kappa, F)= 8(\kappa+6\mathrm{diam}(F))$.
\end{prop}

Before proving Proposition \ref{lem:displaced disk b in bdc}, we first state an immediate corollary. We note that for any $x \in \R^2$ and any $r \in (0,1/2)$, the set $F_{x} = x + [-\frac{1}{2}, \frac{1}{2})^2$ is a bounded fundamental domain containing $B(x,r) \subset \R^2$. Note that $c(\kappa) := \sup_{x \in \R^2}c(\kappa, F_{x}) < 8(\kappa + 12)$ for any $\kappa \geq 0$.
\begin{cor}\label{cor}
Let $f$ be an  area-preserving pseudo-rotation of
$\mathbb{T}^2$. Assume that for $\kappa\geq0$, $\omega \in \R^2$ with $\norm{\omega} < 1/(2c(\kappa))$, a  lift of $f$, denoted by $\tilde{f}$, satisfies that
$\rho(\tilde{f})=\omega \in \mathbb{R}^2\setminus\{(0,0)\}$ and \eqref{eq:bdc} with $v= \omega$ (in particular, if $f$ has bounded mean motion with a bound $\kappa$),  then
$$d_{C^0}(f,\mathrm{Id}_{\mathbb{T}^2}) \leq c(\kappa)^{\frac{1}{2}}\|\omega\|^{\frac{1}{2}}+\max_{z\in\mathbb{T}^2}\mathrm{diam}(f(B(z,c(\kappa)^{\frac{1}{2}}\|\omega\|^{\frac{1}{2}})),$$
where  $B(z,r)$ is the open disc centered at $z$ with radius $r$.
\end{cor}
\begin{proof}
Set $r = (c(\kappa)\norm{\omega})^{1/2} \in (0,1/2)$. Then for any $x \in \R^2$, by $\lambda(B(\pi(x),r)) > r^2 \geq c(\kappa,F_{x})\norm{\omega}$, and by Proposition \ref{lem:displaced disk b in bdc}, we have $f(B(\pi(x),r)) \cap B(\pi(x), r) \neq \emptyset$. The corollary then follows as an immediate consequence.
\end{proof}

\begin{proof}[Proof of Proposition \ref{lem:displaced disk b in bdc}]
Without loss of generality, we can assume that $(0,0) \in F$, for otherwise we can replace $F$ by the unique $\Z^2$ translation of $F$ which contains $(0,0)$.

Let $D \subset \T^2$ be a simple open disc with respect to $F$ such that $f(D) \cap D = \emptyset$, and let $\widetilde{D}$ be the connected component of $\pi^{-1}(D)$ contained in $F$. We denote by $\mathrm{Rec}^+(f)$ the set of positively recurrent points of $f$, i.e. $ x \in \mathrm{Rec}^{+}(f) \iff \liminf_{n \geq 1}d(f^{n}(x), x) = 0$. Note that $\lambda(\mathrm{Rec}^+(f))=\lambda(\T^2)=1$ by Poincar\'e's recurrence theorem. For every $z\in \mathrm{Rec}^+(f)\cap D$, we have  $n_D(z):=\min\{n\geq 1 \mid f^n(z)\in D\} < \infty$,  and we define $f_D(z):=f^{n_D(z)}(z)$.  It is direct to see that $\mathrm{Rec}^+(f)\cap D$ is invariant under $f_D$.

We let $l_D(z)$ be the unique lattice point in $\mathbb{Z}^2$  such that $\tilde{f}^{n_D(z)}(\widetilde{z})\in l_D(z)+\widetilde{D}$, where $\widetilde{z}$ is the unique preimage of $z$ under $\pi$ in $\widetilde{D}$. Then it is clear that $\tilde{f}^{n_D(z)}(\widetilde{z}) = l_D(z) + \widetilde{z}'$,
where $\widetilde{z}'$ is the unique preimage of $f_D(z)$ in $\widetilde{D}$.
By successive applications of the above equality, we see that for every $z\in \mathrm{Rec}^+(f)\cap D$ and every integer $N\geq 0$, 
 \ary \label{lattice esmate0}
 \tilde{f}^{\sum_{i=0}^{N-1}n_D(f_D^i(z)}(\widetilde{z}) = \sum_{i=0}^{N-1}l_D(f_D^i(z)) + \widetilde{z}_N
 \eary 
 where $\widetilde{z}_N$ is the unique preimage of $f_D^{N}(z)$ under $\pi$ in $\widetilde{D}$. Thus
\begin{equation}\label{lattice esmate}
\left\| \tilde{f}^{\,\sum_{i=0}^{N-1}n_D(f_D^i(z))}(\widetilde{z})-\sum_{i=0}^{N-1}l_D(f_D^i(z)) \right\| \leq \mathrm{diam}(F),
\end{equation}
and by $\rho(\tilde{f})= \omega$, we have
\begin{equation}\label{rotation vector}
    \frac{\sum_{i=0}^{N-1}l_D(f_D^i(z))}{\sum_{i=0}^{N-1}n_D(f_D^i(z))}
    -\omega=\delta_N(z)
\end{equation}
for every $z\in \mathrm{Rec}^+(f) \cap D$, where
$\|\delta_N(z)\|\rightarrow 0$ as
$N\rightarrow +\infty$. We also note that the limit $$\lim\limits_{N\rightarrow+\infty}\frac{1}{N}\sum_{i=0}^{N-1}n_D(f_D^i(z))$$ exists  for $\lambda$-a.e. $z\in D$ by Birkhoff's ergodic theorem and sine $f_D$ preserves $\lambda|_{D}$.

For every $z\in\mathbb{R}^2$, we define the  linear form $\phi(z)=\langle z, \frac{\omega}{\norm{\omega}}\rangle,$ and for any $n\geq 1$, the following Birkhoff sum  $$\phi_n(z)=\sum_{i=0}^{n-1}\phi(l_D(f_D^i(z))).$$
For every $r>0$, we define the following strip
$$R_{\omega,r}=\{z\in\mathbb{R}^2\mid d(z, \R \omega) \leq r\}.$$

Let  $\mathrm{Orb}^+(\tilde{f},z)=\{\tilde{f}^n(z)\mid n\geq 0\}$. By the bounded deviation condition \eqref{eq:bdc}, and by (\ref{lattice esmate}), we have
\ary\label{sdf}
\bigcup_{z\in F}\mathrm{Orb}^+(\tilde{f},z)\subset
R_{\omega,\kappa+\mathrm{diam}(F)}
\mbox{ and     } S_D(\tilde{f})\subset
R_{\omega, \kappa+2\mathrm{diam}(F)},
\eary
where
$$S_D(\tilde{f})=\left\{\sum\limits_{i=0}^{k-1}l_D(f_D^i(z))\mid z\in \mathrm{Rec}^+(f)\cap
D, k\geq1\right\}.$$

As $\rho(\tilde{f})=\omega$,  we claim that $\inf\phi(S_D(\tilde{f})) > -\infty$.  Indeed, by (\ref{lattice esmate}), we have that $\inf\phi(S_D(\tilde{f}))\geq \inf\phi(\bigcup_{z\in F}\mathrm{Orb}^+(\tilde{f},z))-\mathrm{diam}(F)$.  Suppose that there exist a sequence $\{z_i\}_{i\geq1}$ in $F$, and a sequence $\{n_i\}_{i\geq1}$ in $\mathbb{N}$ such that $\phi(\tilde{f}^{n_i}(z_i))\rightarrow -\infty$ as $i\rightarrow+\infty$, then $\phi(\frac{\tilde{f}^{n_i}(z_i)-z_i}{n_i})\leq0<\phi(\omega)=\|\omega\|$ for all sufficiently large $i$. This contradicts with the fact that $\rho(\tilde{f})=\omega$ and that $\phi$ is continuous. 

We define 
\begin{equation*}\label{LF}
L_{\tilde{f}}=\max\left\{-\inf\phi(S_D(\tilde{f})),2\right\} \text{ and } R_{\tilde{f}}=\left\{z\in R_{\omega, \kappa+2\mathrm{diam}(F)}\mid |\phi(z)|\leq L_{\tilde{f}}\right\}.
\end{equation*}

It is direct to see that $\mathrm{diam}(F)>1/2$. Then we have
 \begin{equation}\label{lattice}
 \#\left(R_{\tilde{f}}\cap\mathbb{Z}^2\right)< [2(L_{\tilde{f}}+2)]\cdot[2(\kappa+2\mathrm{diam}(F)+2)] < 4(L_{\tilde{f}}+2)(\kappa+6\mathrm{diam}(F)).
 \end{equation}
Indeed, let $R'$ be the union of the disjoint unit squares of the form $[0,1)^2+v$ where $v \in R_{\tilde{f}} \cap\Z^2$. Then $R'$ is contained in a larger closed rectangle $R''$ with  edge lengths  $2(L_{\tilde{f}}+2)$ and $2(\kappa+2{\rm diam}(F) + 2)$. Then (\ref{lattice}) follows from the fact that the left hand side of \eqref{lattice} does not exceed the area $|R'| \leq |R''|$.

For a real number $a$, let $\lfloor a \rfloor=\max\{n\in\Z\mid a-n\geq 0\}$. Set $K_{\tilde{f}}=\lfloor4(L_{\tilde{f}}+2)(\kappa+6\mathrm{diam}(F))\rfloor$.  We claim that, for every $z\in \mathrm{Rec}^+(f)\cap D$, there exists an integer $k_z\in[1,K_{\tilde{f}}]$ such that 
\begin{equation}\label{2kappa}\left|\phi(\sum_{i=0}^{k_z-1}l_D(f_D^i(z)))\right|> L_{\tilde{f}}.\end{equation}

To prove (\ref{2kappa}), we first show that $$\sum_{i=0}^{k-1}l_D(f_D^i(z))\in R_{\omega, \kappa+2\mathrm{diam}(F)}\setminus \{(0,0)\} \text{ for any }k\geq 1.$$ By \eqref{lattice esmate} and \eqref{sdf}, we only need to prove that $\sum_{i=0}^{k-1}l_D(f_D^i(z))\neq (0,0)$ for any $k\geq1$. Assume to the contrary that $k$ is a positive integer with $\sum_{i=0}^{k-1}l_D(f_D^i(z))= (0,0)$. Then the following disks $$\widetilde{D}, \tilde{f}(\widetilde{D}),\cdots,\tilde{f}^{\,\sum_{i=0}^{k-1}n_D(f^i_D(z))}(\widetilde{D})$$ must contain a periodic free disk chain. Then $\tilde{f}$ would have a fixed point by Franks' Lemma (see Proposition \ref{prop:Franks' Lemma}) which contradicts with the fact that $\omega\neq(0,0)$. Moreover, we claim that $$\sum_{i=0}^{k-1}l_D(f_D^i(z))\neq\sum_{i=0}^{k'-1}l_D(f_D^i(z)) \text{ for any }k'>k\geq 1.$$ Indeed, assume to the contrary that the pair $(k,k')$ satisfies the above
property. Then the following disks $$\tilde{f}^{\,\sum_{i=0}^{k-1}n_D(f^i_D(z))}(\widetilde{D}),\tilde{f}^{\,\sum_{i=0}^{k-1}n_D(f^i_D(z))+1}(\widetilde{D}),\cdots,\tilde{f}^{\,\sum_{i=0}^{k'-1}n_D(f^i_D(z))}(\widetilde{D})$$ must contain a periodic free disk chain which again contradicts with the fact $\omega\neq(0,0)$ by Franks' Lemma. Finally, the claim  (\ref{2kappa}) follows from  $\#\left(R_{\tilde{f}}\cap\mathbb{Z}^2\right)\leq K_{\tilde{f}}$ and the pigeonhole principle.

For every $z\in \mathrm{Rec}^+(f)\cap
D$, by the definition of $L_{\tilde{f}}$ and (\ref{2kappa}) above, there exists an integer $k_z\in[1,K_{\tilde{f}}]$ such that
\begin{equation}\label{inqphi}
\phi_{k_z}(z)>L_{\tilde{f}}.
\end{equation}
By \eqref{rotation vector} and the definition of $\phi$,   we have for each $m\in\mathbb{N}$ that
\begin{equation}\label{rotation vector phi}
   \frac{\frac{1}{mK_{\tilde{f}}}\sum\limits_{i=0}^{mK_{\tilde{f}}-1}\phi(l_D(f_D^i(z)))} {\frac{1}{mK_{\tilde{f}}}\sum\limits_{i=0}^{mK_{\tilde{f}}-1}n_D(f_D^i(z))}
    -\|\omega\|=\phi(\delta_{mK_{\tilde{f}}}(z)).
\end{equation}
On the other hand, by \eqref{inqphi} there exist sequences $\{z_j\}_{j=1}^n\subset D$ and  $\{k_j\}_{j=1}^n\subset\mathbb{N}$ such that
\begin{itemize}
\item $n\geq m$; 
\item $(m-1)K_{\tilde{f}}< \sum_{j=1}^n k_j\leq mK_{\tilde{f}}$;
\item $z_1=z$ and $z_{j+1}=f_D^{k_j}(z_j)$ for $1\leq j \leq n$;
\item $\phi_{k_j}(z_j)>L_{\tilde{f}}$.
\end{itemize}
Therefore,
\ary
  \frac{1}{mK_{\tilde{f}}}\sum_{i=0}^{mK_{\tilde{f}}-1}\phi(l_D(f_D^i(z))) &=&\frac{1}{mK_{\tilde{f}}}\left\{\sum_{j=1}^n\phi_{k_j}(z_j)+\sum_{i=0}^{mK_{\tilde{f}}-\sum_{j=1}^nk_j-1}\phi(l_D(f_D^i(z_{n+1})))\right\} \nonumber \\
 \label{lfkf} &>&\frac{L_{\tilde{f}}}{K_{\tilde{f}}}+\frac{1}{mK_{\tilde{f}}}\left\{\sum_{i=0}^{mK_{\tilde{f}}-\sum_{j=1}^nk_j-1}\phi(l_D(f_D^i(z_{n+1})))\right\} \geq  \frac{L_{\tilde{f}}}{K_{\tilde{f}}}  -\frac{L_{\tilde{f}}}{mK_{\tilde{f}}},
\eary
where the last inequality follows from the definition of $L_{\tilde{f}}$.
By \eqref{rotation vector phi}, \eqref{lfkf}, and by letting $m$ tends to infinite, we obtain
\begin{equation*}\liminf_{N\rightarrow+\infty}\frac{1}{N}\sum_{i=0}^{N-1}n_D(f_D^i(z))\geq \frac{L_{\tilde{f}}}{K_{\tilde{f}}}\cdot\frac{1}{\|\omega\|}\geq \frac{L_{\tilde{f}}}{4(L_{\tilde{f}}+2)(\kappa+6\mathrm{diam}(F))}\cdot\frac{1}{\|\omega\|}.\end{equation*}
As $L_{\tilde{f}}\geq 2$, we have $$\frac{L_{\tilde{f}}}{4(L_{\tilde{f}}+2)(\kappa+6\mathrm{diam}(F))}\geq\frac{1}{8(\kappa+6\mathrm{diam}(F))}.$$
Recall that $c(\kappa, F)= 8(\kappa+6\mathrm{diam}(F))$.  Then 
\begin{equation}\label{inqlem2}\liminf_{N\rightarrow+\infty}\frac{1}{N}\sum_{i=0}^{N-1}n_D(f_D^i(z))\geq \frac{1}{c(\kappa, F)\|\omega\|}\quad \text{ for } \lambda\text{-a.e. } z\in D.\end{equation}
By Kac's Lemma, we have
\begin{equation}\label{kaclemma} \int_{D}n_D(z)\mathrm{d}\lambda=\lambda\left(\bigcup_{n\geq0}f^n(D)\right)\leq1.
\end{equation}
 Then Proposition
\ref{lem:displaced disk b in bdc} follows from (\ref{inqlem2}), (\ref{kaclemma}),  and Fatou's Lemma:
\begin{equation*}
\frac{\lambda(D)}{c(\kappa,F)\|\omega\|}=\int_D\frac{1}{c(\kappa,F)\|\omega\|}\mathrm{d}\lambda\leq \int_D\liminf_{N\rightarrow+\infty}\frac{1}{N}\sum_{i=0}^{N-1}n_D(f_D^i(z))\mathrm{d}\lambda\leq1.
\end{equation*}
\end{proof}

\section{Proofs of Theorem  \ref{sup} and Theorem \ref{nonBrjunont}} \label{sec:thmsupnonbrju}

Based on Corollary \ref{cor} in Section \ref{sec:proof}, we are ready to prove Theorem  \ref{sup} and Theorem \ref{nonBrjunont}. Our proof follows closely that of \cite[Theorem 1 and Theorem 2]{A}.

\subsection{Proof of Theorem \ref{sup}}\quad\smallskip

As $f$ is H\"{o}lder with exponent $a$, there is  $C>0$ such that for every $m\in\mathbb{N}$,  
\begin{equation*}\label{holder}
\|f^m(x)-f^m(y)\|\leq C^m \|x-y\|^{a^m}.
\end{equation*}
 \begin{proof}[Proof under condition $(1)$] 
By hypothesis, there exists $\kappa > 0$ such that $f\in\mathcal{C}_{\kappa,\overline{\omega}}$, where $\overline{\omega}=\omega \mod \Z^2 \in \R^2/\Z^2$. 
From the arithmetic condition (\ref{eq:a-super-liouvullen}), 
we obtain a subsequence $\{n_j\}_{j\geq0}$ in $\N$  satisfying
 \begin{equation}\label{termrigiditysuperliouv} 
 C^{n_j}\cdot \norm{n_j\omega}_{\T^2}^{\frac{a^{n_j}}{4}}\rightarrow 0,\quad \text{ as } j\rightarrow +\infty.
 \end{equation}
Since $f \in \cal{C}_{\kappa, \overline{\omega}}$, for each $j \geq 0$ there exists a lift $\tilde{f}_j$ of $f^{n_j}$ such that $\norm{\rho(\tilde{f}_j)} = \norm{\rho(f^{n_j})}_{\T^2} =\norm{n_j\omega}_{\T^2}$. Note that we have $f^{n_j} \in \cal{C}_{\kappa, n_j\overline{\omega}}$.
By applying Corollary \ref{cor} to the pseudo-rotation $f^{n_j}$ and the lift $\tilde{f}_j$,
 and by the facts that:
(1) $2 c(\kappa)^{\frac{1}{2}}\norm{n_j\omega}_{\T^2}^{\frac{1}{2}}<\norm{n_j\omega}_{\T^2}^{\frac{1}{4}}$ for all  sufficiently large $j$; and  (2) $f^{n_j}\in\mathcal{C}_{\kappa,n_j\overline{\omega}}$ for any $j\in \N$,  
we obtain
\begin{equation*}\label{holder i}
d_{C^0}(f^{n_j},\mathrm{Id}_{\mathbb{T}^2}) \leq c(\kappa)^{\frac{1}{2}}\norm{n_j\omega}_{\T^2}^{\frac{1}{2}}+C^{n_j}\left(2c(\kappa)^{\frac{1}{2}}\norm{n_j\omega}_{\T^2}^{\frac{1}{2}}\right)^{a^{n_j}} < \norm{n_j\omega}_{\T^2}^{\frac{1}{4}}+ C^{n_j}\norm{n_j\omega}_{\T^2}^{\frac{a^{n_j}}{4}}.
\end{equation*} 
We conclude the proof by \eqref{termrigiditysuperliouv}.
\end{proof}

\begin{proof}[Proof under condition $(2)$]
We can apply Lemma \ref{lem:Abdc} to $f$, to obtain $A \in \mathrm{SL}(2,\Z), L \in \N$ and 
$$f' := T_A f^{L} T_A^{-1} \in \mathcal{D}_{\kappa', (\gamma, 0), (1,0)}$$
 for $\kappa' = \norm{A} \kappa$, and some $\gamma \in \R/\Z$ with $(\gamma,0) = LT_A(\overline{\omega})$.
By  Remark \ref{ftof'}, it is enough to show that $f'$ is $C^0$-rigid. By \eqref{eq:a-super-liouvullen}, we obtain
\aryst
\liminf_{n\rightarrow+\infty} n^{-1}a^{n}\ln \norm{n\rho(f')}_{\T^2} \leq \liminf_{n\rightarrow+\infty} n^{-1}a^{n}\ln (L \norm{A}\norm{n\overline{\omega}}_{\T^2}) = -\infty.
\earyst
Thus without loss of generality, it is enough to prove Theorem \ref{sup} assuming that  $\rho(f) = (\gamma, 0)$ for some $\gamma \in (\R\setminus \Q)/\Z $, and $f \in \mathcal{D}_{\kappa, (\gamma, 0), (1,0)}$ for some $\kappa > 0$. We now proceed with the proof under such assumption.

Note that for any $n \geq 0$, we have $f^{n} \in \mathcal{D}_{\kappa, \pi(n\gamma,0), (1,0)}$; and there exists a lift of $f^{n}$, denoted by $\tilde{f}_n$, such that $\rho(\tilde{f}_n)$ is parallel to $(1,0)$ and $\norm{\rho(\tilde{f}_n)} = \norm{n \rho(f)}_{\T^2} = \norm{n\gamma}_{\T}$. Thus we can apply \eqref{eq:a-super-liouvullen} to find sequence $\{n_j\}_{j \geq 0} \subset \N$ as in the proof above under condition (1). Then we apply Corollary \ref{cor} to $(f^{n_{j}}, \tilde{f}_{n_j})$ in place of $(f, \tilde{f})$ to conclude the $C^{0}$-rigidity of $f$ as in the proof under (1).
 
\end{proof}

\subsection{Proof of Theorem \ref{nonBrjunont}}\label{proof of nb}\quad\smallskip

Without loss of generality, we will assume that $r \geq 2$ is finite. The statement for $r = \infty$ follows as an immediate consequence.
To prove Theorem \ref{nonBrjunont}, we need the following two results in \cite{A}: the first one is the growth of the denominators of a non-Brjuno type  irrational number, and the second one is a growth gap theorem (see also \cite{PolterovichSodin}) for area-preserving $C^2$-diffeomorphisms of compact surfaces.

\begin{lem}[Lemma 3.2, \cite{A}]\label{lem}
 Suppose that $\alpha\in \mathbb{R}\setminus\mathbb{Q}$ is an irrational non-Brjuno number. For any $H>1$, there exists a subsequence $\{q_{n_j}\}_{j\geq1}$ of the sequence $\{q_n(\omega)\}_{n\geq0}$ such that $q_{n_{j+1}}\geq H^{q_{n_j}}$ and there exists an infinite set $\mathcal{J}$ such that 
\begin{equation}\label{ac}
\text{for any } j\in\mathcal{J}, \quad \norm{q_{n_j}\alpha}_{\T}<e^{-\frac{q_{n_j}}{j^2}}.\end{equation}
\end{lem}

%
%

\begin{thm}[Theorem C, \cite{A}]\label{h}
Let $S$ be a compact orientable surface.
For any compact subset $K\subset \mathrm{Diff}^2(S, \mathrm{Vol})$, there exist $0<\theta<1$ and integer $H>0$ satisfying the following property: let
$\{Q_n\}_{n\geq 0} \subset \N$ be a sequence such that $Q_0\geq H$ and $Q_n\geq H^{Q_{n-1}}$ for any $n \geq 1$, if for some $f\in K$ there exists $n\geq 0$ such that 
\begin{equation}\label{hypf}
\frac{1}{Q_n}\ln \|Df^{Q_n}\|>\theta^n,
\end{equation}
then $f$ has a hyperbolic periodic point.
\end{thm}

\begin{proof}[Proof under condition $(1)$]
Let  $\theta$ and  $H$ be given by Theorem \ref{h} for $S = \T^2$ and $K = \{f\}$.  By \eqref{eq:super-liouvullen}, we can obtain a strictly increasing subsequence $\{n_j\}_{j\geq0} \subset \N$  satisfying
\begin{equation}\label{sac}
\norm{n_{j}\omega}_{\T^2}\leq e^{-n_{j}}\quad\text{for every }j\geq0.
\end{equation}
Moreover, after replacing $\{n_j\}_{j\geq0}$ by one of its subsequences, we can assume  that $n_{0}\geq H$ and $n_{j+1}\geq H^{n_{j}}$ for any $j\geq0$.

 Since $f$ is an irrational pseudo-rotation on $\T^2$, in particular, $f$ has no periodic points, Theorem \ref{h} implies that 
\begin{equation}\label{snb}
\|D f^{n_{j}}\|\leq e^{\theta^{j}n_{j}}<e^{\frac{n_{j}}{j^3}}\quad \text{ when } j \text{ is sufficiently large.}
\end{equation}

For each $n_j$, we choose a lift of $f^{n_j}$, denoted by $\tilde{f}_j$, with $\rho(\tilde{f}_j) = \norm{n_j \omega}_{\T^2}$ as in the proof of Theorem \ref{sup} under (1).
By Corollary \ref{cor}, (\ref{sac}), (\ref{snb}), we obtain for all sufficiently large $j$  that 
$$d_{C^0}(f^{n_{j}},\mathrm{Id}_{\mathbb{T}^2})<c(\kappa)^{\frac{1}{2}}(e^{-\frac{n_{j}}{2}}+2e^{-\frac{n_{j}}{2}}e^{\frac{n_{j}}{j^3}})\leq e^{-\frac{n_{j}}{4}}.$$
Recall that we have assumed $r \in \N_{\geq 2}$.
By (\ref{snb}), for some constant  $C_r > 0$ independent of $j$,
 $$\|D^r f^{n_{j}}\|< e^{C_r\frac{n_{j}}{j^3}}.$$ 
Therefore, by the convexity (Hadamard-Kolmogorov) inequality \cite{Kol}, we obtain
$$d_{C^{r-1}}(f^{n_{j}},\mathrm{Id}_{\mathbb{T}^2})=o(e^{-\frac{n_{j}}{j^{3}}}).$$
\end{proof}

\begin{proof}[Proof under condition $(2)$]
We again apply Lemma \ref{lem:Abdc} to $f$ and obtain $A, L$ and $f' = T_A f^{L} T_A^{-1} \in  \mathcal{D}_{\|A\|\kappa, \pi(\ell\beta,0), (1,0)}$, where $\beta = \cF(\overline{\omega})$. By  Remark \ref{ftof'}, it is enough to show that $f'$ is $C^{r-1}$ rigid. 

We denote by $\{q_n\}_{n\geq0}$ the sequence of denominators for $\beta$. Let  $\theta$ and  $H$ be given by Theorem \ref{h} for $S = \T^2$ and $K = \{f'\}$. After applying Lemma \ref{lem}  to $H$ and $\beta$, and upon passing to a subsequence, we can choose $\{q_{n_j}\}_{j\geq0}$ so that $q_{n_0} \geq H$, $q_{n_{j+1}} \geq H^{q_{n_j}}$ for any $j \geq 0$, and
\begin{equation*}
\norm{\rho((f')^{n_{j}})}_{\T^2}=\norm{q_{n_{j}}\ell \beta}_{\T}\leq \ell e^{-\frac{q_{n_{j}}}{j^2}}\quad\text{for every }j\geq0.
\end{equation*}
Then we can conclude the proof of (2) by following the argument in the proof of Theorem \ref{nonBrjunont} under condition (1).
\end{proof}

\section{Proofs of Theorem \ref{cortoplinear2} and  Theorem \ref{cortoplinearminimal}}
\label{sec:thmcorto}
Since $f \in \cal{C}_{\overline{\omega}}$, $f$ is of bounded mean motion. Thus
there exists a constant $\kappa > 0$ such that $f \in \cal{C}_{\kappa, \overline{\omega}}$, and let $\tilde{f}$ be an arbitrary lift of $f$ to $\R^2$. We note that such $f$ is topologically transitive by \cite[Theorem D]{J}. In the following, for each $x \in \T^2$, we denote by $\tilde{x}$ an arbitrary element of $\pi^{-1}(x)$.
\subsection{Proof of Theorem \ref{cortoplinear2}}\quad\smallskip
\begin{lem}\label{gfisuniformlybounded}
We have $G_{f} \subset \cup_{\alpha \in \R^2/\Z^2} \cal{C}_{2\kappa, \alpha}$.
\end{lem}
\begin{proof}
We fix an arbitrary $g \in G_{f}$, and let $\tilde{g}$ be an arbitrary lift of $g$ to $\R^2$. Note that
\aryst
\rho(\tilde{g}) = \int_{\T^2} (\tilde{g}(\tilde{x}) - \tilde{x}) d\lambda(x).
\earyst
 For any integer $n > 0$, we define a continuous function by
\aryst
\Phi_n : \T^2 \to \R^{2}, \quad x \mapsto \tilde{g}^n(\tilde{x}) - \tilde{x}.
\earyst
If there exist $x_0, x_1 \in \T^2$ such that $\norm{\Phi_n(x_0) - \Phi_n(x_1)} > 2\kappa$, then by continuity, there exist open neighbourhoods  $U_i \subset \T^2$ of $x_i$ for $i=0,1$, such that for any $y_i \in U_i$, $i=0,1$, we have $\norm{\Phi_n(y_0) - \Phi_n(y_1)} > 2\kappa$. By the transitivity of $f$, there exist $z \in U_0$, and an integer $m > 0$ such that $f^{m}(z) \in U_1$. Fix an arbitrary $\tilde{z} \in \pi^{-1}(z)$. By commutativity, we have 
\aryst
\Phi_n(z) - \Phi_n(f^{m}(z)) &=& \tilde{g}^n(\tilde{z}) - \tilde{z} - (\tilde{g}^n(\tilde{f}^{m}(\tilde{z})) - \tilde{f}^{m}(\tilde{z})) \\
&=& [\tilde{f}^{m}(\tilde{z}) -  \tilde{z} - m\rho(\tilde{f}) ]-[\tilde{f}^{m}(\tilde{g}^n(\tilde{z})) - \tilde{g}^n(\tilde{z}) - m\rho(\tilde{f})]
\earyst
By $f \in \cal{C}_{\kappa, \overline{\omega}}$, the above two terms in the brackets have norms at most $\kappa$. While by our choices of $z,m,U_0$ and $U_1$, we have $\norm{\Phi_n(z) - \Phi_n(f^{m}(z))} > 2\kappa$. This is a contradiction.
Thus for any $x_0, x_1 \in \T^2$, we have $\norm{\Phi_n(x_0) - \Phi_n(x_1)} \leq 2\kappa$. This implies 
\aryst
\left\|\Phi_n(x) - \int_{\T^2} \Phi_n(y) d\lambda(y)\right\| \leq \int_{\T^2} \norm{\Phi_n(x) - \Phi_n(y)}d\lambda(y) \leq 2\kappa.
\earyst
Our lemma follows from the observation that $\int_{\T^2} \Phi_n(y) d\lambda(y) = n\rho(\tilde{g})$.
\end{proof}
For any $g \in G_{f}$, we define a continuous function $D_{g} : \T^2 \to \R^2/\Z^2$,
\aryst D_{g}(x) = (\tilde{g}(\tilde{x}) - \tilde{x}) \mod \Z^2,
\earyst
where $\tilde{g}$ is a lift of $g$.
We choose a point $x_0 \in \T^2$ so that $x_0$ has a dense orbit under $f$. We define two maps as follows
\aryst
\phi_1 : G_{f} \to \R^2/\Z^2, &&\quad \phi_1(g) = \rho(g) = \int D_{g}(x) d\lambda(x), \\
\phi_2 : G_{f} \to \T^2, &&\quad \phi_2(g) = g(x_0).
\earyst
By definition, $G_f$ is a topological group under the compositions in ${\rm Homeo}_{*}(\T^2)$ with respect to the $C^0$ topology of ${\rm Homeo}_{*}(\T^2)$. It is clear that maps $\phi_1$ and $\phi_2$ are continuous. By definition, the range of $\phi_2$ is dense. Since $\overline{\omega}$ is totally irrational, we note that the range of $\phi_1$ is also dense: this follows from $\{f^n\}_{n \in \Z} \subset G_f$, and $\phi_1(f^{n}) = n\rho(f) = n \overline{\omega}$. Moreover, the map $\phi_1$ is a group homomorphism. Indeed, for any $g_1, g_2 \in G_f$, we have
\aryst
\phi_1(g_1g_2) = \int D_{g_1g_2} d\lambda 
= \int (D_{g_1} \circ g_2 + D_{g_2}) d\lambda = \int (D_{g_1} + D_{g_2}) d\lambda = \phi_1(g_1) + \phi_1(g_2).
\earyst

The key point is the following lemma.
\begin{lem}\label{trivialkernel}
If $g \in G_f$ satisfies $\rho(g) = 0$, then $g = {\rm Id}_{\T^2}$.
\end{lem}
\begin{proof}
By Theorem \ref{nonBrjunont}, there exists a sequence $\{n_j\}_{j \geq 1}$ in $\N$ such that $f^{n_j} \to {\rm Id}_{\T^2}$ in the $C^{r-1}$ topology. Thus the maps $\{ gf^{n_j} \}_{j \geq 1}$ have uniformly bounded $C^{r-1}$ norms, and $\rho(gf^{n_j}) = \rho(f^{n_j}) \to 0$ as $j \to \infty$. Moreover, since $gf^{n_j} \in G_f$ for all $j \geq  1$, by Lemma \ref{gfisuniformlybounded} and Corollary \ref{cor}, we see that
\aryst
d_{C^0}(gf^{n_j},{\rm Id}_{\T^2}) \to 0, \mbox{ as } j \to \infty.
\earyst
Since we clearly have $d_{C^0}(f^{n_j}, {\rm Id}_{\T^2}) \to 0$ as $j \to \infty$, thus $g = {\rm Id}_{\T^2}$. 
\end{proof}
We have the following immediate corollary.
\begin{cor}\label{phi1phi2injections}
Both $\phi_1$ and $\phi_2$ are injections.
\end{cor}
\begin{proof}
Since $\phi_1$ is a group homomorphism, it is enough to show that ${\rm Ker}(\phi_1) = \{{\rm Id}_{\T^2}\}$. This is precisely the content of Lemma \ref{trivialkernel}.

Suppose that there exist $g_1,g_2\in G_f$ such that $\phi_2(g_1) = \phi_2(g_2)$. Then $g_2^{-1}g_1(x_0) = x_0$. Since $g_2^{-1}g_1 \in G_f$, by Lemma \ref{gfisuniformlybounded}, $g_2^{-1}g_1$ is a pseudo-rotation. This implies that $\rho(g_2^{-1}g_1) = 0$, and by Lemma \ref{trivialkernel}, $g_1 = g_2$.
\end{proof}
In particular, $G_f$ is isomorphic to a subgroup of $\R^2/\Z^2$.
By Theorem \ref{nonBrjunont} and \cite[Chap XII, (3.2)]{H}, $G_f$ is also uncountable.

To finish the proof, we first note that the pre-compactness of $G_f$ in the $C^0$ topology is clearly necessary for $f$ to be topologically conjugate to a translation. Thus it suffices to prove the other implication. Assume that $G_f$ is compact in the $C^0$ topology. 
Then the image of $\phi_1, \phi_2$ are both dense and compact, and thus $\phi_1$,$\phi_2$ are surjections.
By Corollary \ref{phi1phi2injections}, $\phi_1, \phi_2$ are also injections. By the compactness of $G_f$, we conclude that $\phi_1,\phi_2$ are homeomorphisms. 

We let $h = \phi_1 \phi_2^{-1}$. It is direct to verify the relation
\aryst
T_{\rho(f)}h = h f.
\earyst

\subsection{Proof of Theorem \ref{cortoplinearminimal}}\quad\smallskip

We follow the same lines as in the proof of Theorem \ref{cortoplinear2} before Lemma \ref{trivialkernel} for $G^{0}_f$ in place of $G_f$. The proof of Lemma \ref{trivialkernel} under the conditions of Theorem \ref{cortoplinearminimal} is the following. 
Let $g \in G^0_f$ satisfy that $\rho(g) = 0$. Then by \cite{F3}, $g$ has a fixed point $z$. By commutativity, every point in the orbit of $z$ under $f$ is fixed by $g$. Since by hypothesis $f$ is  minimal, thus $g = {\rm Id}_{\T^2}$. This completes the proof of Lemma \ref{trivialkernel}.
By definition, $G^{0}_f$ is closed in the $C^0$ topology. Thus if $G^{0}_f$ is pre-compact then it is compact. We conclude the proof following the same lines as in the proof of Theorem \ref{cortoplinear2}.

\section{Proof of Theorem \ref{thm example}}\label{secthm4}
It is enough to prove the following proposition. Theorem \ref{thm example} will then be a corollary.
\begin{prop}\label{prop example}
There exists an area-preserving and minimal pseudo-rotation $f \in \diff^{\infty}(\T^2)$ which has bounded mean motion, and super-Liouvillean rotation vector, and satisfies the following: for any $\varepsilon > 0$, there exist two points $x, y \in \T^2$ with $d(x,y) < \varepsilon$, and an integer $N > 0$ such that $d(f^{N}(x), f^{N}(y)) \geq \frac{1}{1000}$.
\end{prop}

\begin{proof}
Let us denote $\Gamma = (\Q \times \R)/ \Z^2 \cup (\R \times \Q)/\Z^2 \subset \T^2$.
We will construct $h_n \in \Diff^{\infty}(\T^2, {\rm Vol}) \cap \mathrm{Homeo_*}(\T^2)$ and $\omega_n = (\omega_{n,1}, \omega_{n,2}) = q_n^{-1}\hat{\omega}_n \in \Q^2$ with $\hat{\omega}_n  \in \Z^2$, $q_n \in \N$, $q_n> 10^n$ for each $n \geq 1$, such that the following holds for each $n \geq 1$:
\enmt
\item[$(a1)_{n}$] For some lift of $h_n$ to $\R^2$, denoted by $\tilde{h}_n$, we have $d_{C^0}(\tilde{h}_n, \Id_{\R^2}) < 2^{-n}$. In particular, by setting $H_n := h_{1} \cdots h_n \in \Diff^{\infty}(\T^2, \rm Vol )$, then $H_n \in \mathrm{Homeo_*}(\T^2)$, and the map $\tilde{H}_n := \tilde{h}_1 \cdots \tilde{h}_n$ is a lift of $H_n$, and we have $d_{C^0}(\tilde{H}_n, \Id_{\R^2}) \leq \sum_{i=1}^{n}d_{C^0}(\tilde{h}_{i}, \Id_{\R^2}) < 1 - 2^{-n}$;\footnote{\,We stress that although the distances between $h_n$ and ${\rm Id_{\T^2}}$ are summable, $\{H_n\}$ does not converge. See $(a2)$. }
\item[$(a2)_{n}$] There exist $x_n, y_n \in \T^2$ with $x_n-y_n \notin \Gamma$ such that $$d(x_n, y_n)  <  10^{-2n} \,\text{   and   }\, d(H_n(x_n), H_n(y_n)) > \frac{1}{1000};$$
\item[$(a3)_{n}$] Set $f_n := H_n T_{\omega_n}(H_n)^{-1}  \in \Diff^{\infty}(\T^2, \rm Vol ) \cap Homeo_*(\T^2)$. There exist $x^{(n)}, y^{(n)} \in \T^2$, $m_n \in \N$ such that $d(x^{(n)}, y^{(n)}) < 2^{-n}$ and $d(f^{m_n}_n(x^{(n)}), f^{m_n}_n(y^{(n)})) > \frac{1}{1000}$;
\item[$(a4)_{n}$] For any $(k_1,k_2,k_3) \in \{-n, \cdots, n\}^{3} \setminus \{(0,0,0)\}$, we have $k_1 \omega_{n,1} + k_2 \omega_{n,2} + k_3 \neq 0$;
\item[$(a5)_{n}$] For any $z \in \T^2$, the set $\{f_n^{k}(z)\}_{k \in \Z}$ is $2^{-n}$-dense in $\T^2$.
\eenmt
Here we say that a set $K \subset \T^2$ is $\sigma$-dense for some $\sigma > 0$, if for any $x \in \T^2$ there exists $y \in K$ such that $d(x,y) < \sigma$.

Note that $(a1)_{n}$ and $(a3)_{n}$ imply the following:  the map $F_n := \tilde{H}_nT_{\omega_n} \tilde{H}_n^{-1}$ is a lift of $f_n$, and for  any integer $k \geq 1$ we have
\ary
\sup_{z \in \R^2}\norm{F_n^{k}(z) - z - k \omega_n} &=& \sup_{z \in \R^2}\norm{ \tilde{H}_n (\tilde{H}_n^{-1}(z) + k\omega_n) -\tilde{H}_n(\tilde{H}_n^{-1}( z)) - k \omega_n} \nonumber \\
&=&\sup_{z \in \R^2}\norm{ \tilde{H}_n (z + k\omega_n) -\tilde{H}_n(z) - k \omega_n}  \nonumber \\
&\leq& 2d_{C^0}(\tilde{H}_n, \Id_{\R^2})  < 10. \label{PartI3term 111}
\eary
Moreover, we let $\epsilon_n > 0$ be a sufficiently small real number so that 
for any $f \in \mathrm{Homeo}(\T^2)$ satisfying $d_{C^0}(f, f_n) < \epsilon_n$, for any $\omega' = (\omega'_1, \omega'_2) \in \R^2$ satisfying $\norm{\omega' - \omega_n} < \epsilon_n$, and any $F' \in \mathrm{Homeo}(\R^2)$ satisfying $d_{C^0}(F', F_{n}) < \epsilon_n$, we have
\ary 
&&\label{PartI3term 1} \norm{(F')^{k}(z) - z - k \omega'} < 10, \quad \forall z \in \R^2, 1 \leq k \leq n, \\
&&  \label{PartI3term 2} d(f^{m_n}(x^{(n)}), f^{m_n}(y^{(n)})) > \frac{1}{1000}, \\
&& \label{PartI3term 3}  k_1 \omega'_1  + k_2 \omega'_2 + k_3 \neq 0,  \forall (k_1,k_2,k_3) \in \{-n, \cdots, n\}^{3} \setminus \{(0,0,0)\}, \\
&& \label{PartI3term 4} \{f^{k}(z)\}_{k \in \N} \mbox{ is $2^{-n+1}$-dense in $\T^2$ for any } z \in \T^2.
\eary
We can see that such $\epsilon_n$ exists by \eqref{PartI3term 111}, $(a3)_n$, $(a4)_n$ and $(a5)_n$.
Without loss of generality, we assume that $\epsilon_k > \epsilon_{k+1}$ for any  $k \geq 1$.

We will further assume that for each integer $n \geq 1$, the following is true:
\enmt
\item[$(a6)_n$] we have $d_{\Diff^{\infty}(\T^2)}(f_{n+1}, f_n), d_{C^0}(F_{n+1}, F_{n}), \norm{\omega_{n+1} - \omega_{n}} < 2^{-n}\epsilon_{n}$.
\eenmt

We let $h_1 = \Id_{\T^2}$ and $\omega_{1} = q_1^{-1}\hat{\omega}_1 = (\frac{1}{100}, \frac{1}{10})$, where $q_1 = 100$ and $\hat{\omega}_1 = (1,10)$. It is direct to verify $(a1)_1, \cdots, (a5)_1$.
Assume that we have constructed $(h_i, \omega_i, q_i, \hat{\omega}_i)$ for any $1 \leq i \leq n$, satisfying $(a1)_i - (a5)_i$ for any $1 \leq i \leq n$, and $(a6)_i$ for any $1 \leq i \leq n-1$. We construct $(h_{n+1}, \omega_{n+1}, q_{n+1}, \hat{\omega}_{n+1})$ as follows.

Let $f_n, F_n, H_n$, $\omega_n$, $x_n$, $y_n$, $x^{(n)}$, $y^{(n)}$, $m_n$, $\epsilon_n$ be given as above.  We first note the following lemma.
\begin{lem} \label{PartI3lem h n}
For any integer $q \geq 2$, any $x \neq y \in \T^2$ such that $x-y \notin \Gamma$, any $\sigma > 0$, there exists
 $h \in \diff^{\infty}(\T^2, {\rm Vol}) \cap Homeo_*(\T^2)$ such that 
\enmt
\item[$(1)$] $h$ commutes with both $T_{(\frac{1}{q}, 0)}$ and $T_{(0, \frac{1}{q})}$;
\item[$(2)$] There exist $x', y' \in \T^2$ such that $d(x', y') < \sigma$,  $x'-y' \notin \Gamma$,  and
\aryst
 &&h(z') = z, \quad \mbox{for } z = x,y, \\
&&d(hT_{(\frac{1}{2q}, 0)}(x'), hT_{(\frac{1}{2q}, 0)}(y') ) < \sigma;
\earyst
\item[$(3)$] There exists $\tilde{h}$, a lift of $h$, satisfying $d_{C^{0}}(\tilde{h}, \Id_{\R^2}) \leq 3d(x,y) + \frac{2}{q}$.
\eenmt
\end{lem}
\begin{proof}
Since by definition $\T^2 = (\R / \Z)^2$, for any intervals $I_1, I_2 \subset \R$, we can naturally identify $I_1 \times I_2$ with a subset of $\T^2$. We will first consider the case where $x = (0,0)$ and $y = (y_1,y_2) \in (0,\frac{1}{2})^2 \setminus \{(0,0)\}$. In this case, we have $y_1+y_2 \leq 2d(x,y)$. Moreover, there exists a unique $k \in \Z$ with $-\frac{k}{q} < y_2 < -\frac{k-1}{q}$, such that $ T_{(0, \frac{k}{q})}(y) \in (0,1) \times (0, \frac{1}{q})$ and $|\frac{k}{q}| \leq \norm{y}$.

Let $\varphi_1 : \R \to (-1,1)$ be a $\frac{1}{q}-$periodic $C^{\infty}$ function such that:
(1) $\varphi_1(y_2) \in (y_1 - \frac{1}{q}, y_1)$; (2) $\varphi_1(0) = 0$; (3) and $\norm{\varphi_1} \leq \norm{y}$. Such $\varphi_1$ exists by $y_2 \notin \frac{1}{q}\Z$ since $x-y \notin \Gamma$. Define $\tilde{\Phi}_1 \in \diff^{\infty}(\R^2, {\rm Vol})$ as
\aryst
\tilde{\Phi}_1(w_1,w_2) = (w_1 + \varphi_1(w_2), w_2 - \frac{k}{q}).
\earyst
Then we have $d_{C^0}(\tilde{\Phi}_1, \Id_{\R^2}) \leq |\frac{k}{q}| + \norm{\varphi_1} \leq 2\norm{y}$.
It is direct to see that $\tilde{\Phi}_1$ is a lift of a map $\Phi_1 \in \Diff^{\infty}(\T^2, \rm Vol) \cap Homeo_*(\T^2)$, and
\aryst
\Phi^{-1}_1(y) = (y_1 -\varphi_1(y_2 ), y_2 + \frac{k}{q})  \in (0,\frac{1}{q}) \times (0, \frac{1}{q}), \ \mbox{ and } \
\Phi^{-1}_1(x) = ( 0, \frac{k}{q}).
\earyst

Let $\varphi_2 : \R \to (-1,1)$ be a $\frac{1}{q}-$periodic $C^{\infty}$ function such that: (1) $\varphi_2(0) =  \frac{k}{q}$; (2) $\varphi_2(y_1 -\varphi_1(y_2)) = 0$; (3) and $\norm{\varphi_2} \leq |\frac{k}{q}| \leq \norm{y}$.
Define $\tilde{\Phi}_2 \in \diff^{\infty}(\R^2, {\rm Vol})$ as
\aryst
\tilde{\Phi}_2(w_1,w_2) = (w_1, w_2 + \varphi_2(w_1)).
\earyst
Then we have $d_{C^0}(\tilde{\Phi}_2, \Id_{\R^2}) \leq \norm{y}$.
It is direct to see that $\tilde{\Phi}_2$ is a lift of a map $\Phi_2 \in \Diff^{\infty}(\T^2, \rm Vol) \cap Homeo_*(\T^2)$, and
\aryst
\Phi_2^{-1}\Phi_1^{-1}(y) =  \Phi_1^{-1}(y)  \in (0,\frac{1}{q}) \times (0, \frac{1}{q}), \ \mbox{ and } \
\Phi_2^{-1}\Phi_1^{-1}(x) = ( 0, 0).
\earyst

By direct construction, we can find: a map $g \in \diff^{\infty}(\T^2, {\rm Vol})  \cap Homeo_*(\T^2)$ satisfying $g([0,\frac{1}{q})^2) = [0,\frac{1}{q})^2$, commuting with $T_{(\frac{1}{q}, 0)}$ and $T_{(0, \frac{1}{q})}$; and $y' \in  [0,\frac{1}{4q})^2 \setminus \Gamma$ satisfying $\norm{y'} < \sigma$, such that :
\aryst
g(0,0) = (0,0), \quad g(y') = \Phi_2^{-1}\Phi_1^{-1}(y), \quad d(g(\frac{1}{2q}, 0), g(y'+(\frac{1}{2q},0))) < \norm{\Phi_1\Phi_2}_{C^1}^{-1}\sigma.
\earyst
It is direct to see that there exists a lift of $g$, denoted by $\tilde{g}$, such that  
\aryst
d_{C^0}(\tilde{g}, \Id_{\R^2}) \leq \frac{2}{q}.
\earyst
Let $\tilde{h} = \tilde{\Phi}_1\tilde{\Phi}_2 \tilde{g}$  and $h = \Phi_1 \Phi_2 g$.
Note that $h \in  \diff^{\infty}(\T^2, {\rm Vol})  \cap Homeo_*(\T^2)$, and 
\aryst
d_{C^0}(\tilde{h}, \Id_{\R^2}) \leq d_{C^0}(\tilde{\Phi}_1, \Id_{\R^2}) + d_{C^0}(\tilde{\Phi}_2, \Id_{\R^2}) + d_{C^0}(\tilde{g}, \Id_{\R^2}) \leq 3\norm{y} + \frac{2}{q}.
\earyst
It is then direct to verify (1)-(3).

We now consider the case for an arbitrary $x \in \T^2$. Without loss of generality, we can assume that $y-x \in (0,\frac{1}{2})^2$. The other cases are handled by symmetrical constructions.

In this case, we apply our lemma to $(0,y-x)$ in place of $(x,y)$, to obtain $(h', x'', y'')$ satisfying (1)-(3) in place of $(h,x',y')$. Then we let
\aryst
h = T_{x} h' T_{-x}, \quad x' = x + x'', \quad y' = x + y''.
\earyst
Then (1) and (2) are clear for $(h,x',y')$. We verify (3) by noting that
\aryst
d_{C^0}(\tilde{h}, \Id_{\R^2}) = d_{C^0}(\tilde{h'}\tilde{T}_{-x}, \tilde{T}_{-x}) = d_{C^0}(\tilde{h'}, \Id_{\R^2}) \leq 3d(0, y-x) + \frac{2}{q} = 3d(x,y) + \frac{2}{q}.
\earyst
\end{proof}
We recall that $\omega_n = q_n^{-1} \hat{\omega}_n$. By $(a2)_n$ ($x_n-y_n \notin \Gamma$), we can apply Lemma \ref{PartI3lem h n} to $q = q_n$, $x = x_n$, $y = y_n$ and $\sigma =$ $ \max(1, \norm{H_n}_{C^1})^{-1} 10^{-2n-4}$ to obtain $(h_{n+1}, x_{n+1}, y_{n+1}) := (h, x', y')$. Then $h_{n+1} \in  \diff^{\infty}(\T^2, {\rm Vol})  \cap Homeo_*(\T^2)$; $x_{n+1}-y_{n+1} \notin \Gamma$ and $d(x_{n+1}, y_{n+1}) < 10^{-2n-2}$; $h_{n+1}$ commutes with $T_{\omega_{n}}$; and there exists $\tilde{h}_{n+1}$, a lift of $h_{n+1}$, satisfying
\aryst
d_{C^0}(\tilde{h}_{n+1}, \Id_{\R^2}) \leq 3d(x_n, y_n) + \frac{2}{q_n} < 2^{-n-1},   \qquad
\earyst
where the last inequality of the above follows from $(a2)_n$ and our hypothesis that $q_n > 10^{n}$. This verifies $(a1)_{n+1}$.
Moreover by Lemma \ref{PartI3lem h n} (2), we have $h_{n+1}(z_{n+1}) = z_{n}$ for $z = x,y$. By $(a1)_n$, we have $H_{n+1} = H_{n} h_{n+1}$. By  $(a2)_n$, we have $d(H_{n+1}(x_{n+1}), H_{n+1}(y_{n+1})) = d(H_n(x_n), H_n(y_n)) > \frac{1}{1000}$. Thus the above discussions verify $(a2)_{n+1}$.

Let 
\aryst
z^{(n+1)}& :=& H_{n+1} T_{(\frac{1}{2q_n}, 0)}(z_{n+1}) \quad \mbox{for } z = x,y.
\earyst
By Lemma \ref{PartI3lem h n} (2), we see that 
\aryst
d(x^{(n+1)}, y^{(n+1)}) \leq \norm{H_n}_{C^1}d(h_{n+1}T_{(\frac{1}{2q_n}, 0)}(x_{n+1}), h_{n+1}T_{(\frac{1}{2q_n}, 0)}(y_{n+1})) < 10^{-n-1}.
\earyst
We thereby verify the first part of $(a3)_{n+1}$.

For any $\gamma \in \R^2$ , we set
\aryst
G^{\gamma}_n := H_{n+1}  T_{(-\frac{1}{2q_n}, 0)+ \gamma} H_{n+1}^{-1}.
\earyst
By Lemma \ref{PartI3lem h n} (2), we have
\aryst
G^{(0,0)}_{n}(z^{(n+1)}) = H_{n+1} (z_{n+1}) = H_n(z_n) \quad \mbox{for }  z=x,y.
\earyst

Then by continuity  and $(a2)_n$, there exists $\kappa > 0$ such that for any $\gamma \in \R^2$ with $\norm{\gamma} < \kappa$, we have 
\ary \label{PartI3d G G}
d(G^{\gamma}_n(x^{(n+1)}), G^{\gamma}_n(y^{(n+1)})) > \frac{1}{1000}.
\eary
Without loss of generality, we also assume that
\ary \label{PartI3kappa 2 n 1}
\kappa < 2^{-n-1}\norm{H_{n+1}}_{C^1}^{-1}.
\eary

Set $\omega_{n+1} = \omega_n + \beta_{n+1}$ for some $\beta_{n+1} \in \Q^{2} \setminus \{(0,0)\}$ of the form
\aryst
\beta_{n+1} = (\beta_{n+1,1}, \beta_{n+1,2}) := \frac{1}{q_{n}r_{n+1}}(1,v).
\earyst
Here $v \in \N$ satisfies $v > 100\max(\kappa^{-1}, n+1)$; and let $r_{n+1} \geq 100(n+1)^2\kappa^{-1}v$ be a large integer, to be determined later.
Note that for any $(k_1,k_2) \in \{-(n+1), \cdots, (n+1)\}^2\setminus \{(0,0)\}$, we have $k_1\beta_{n+1,1} + k_2\beta_{n+1,2}  \neq 0$, and
\aryst
 \norm{ k_1\beta_{n+1,1} + k_2\beta_{n+1,2} }_{\T} &<& \frac{20(n+1)v}{q_{n}r_{n+1}} < \frac{1}{2q_{n}}.
\earyst

For any $(k_1,k_2, k_3) \in \{-(n+1), \cdots, (n+1)\}^3\setminus \{(0,0, 0)\}$, 
we have $q_{n}(k_1\omega_{n,1} + k_2\omega_{n,2} ) \in \Z$, and hence
\aryst
k_1 \omega_{n+1,1} + k_2 \omega_{n+1,2} + k_3 = (k_1 \omega_{n,1} + k_2 \omega_{n,2}) + ( k_1 \beta_{n+1,1} + k_2 \beta_{n+1,2} +  k_3) \neq 0.
\earyst
This verifies $(a4)_{n+1}$.

By choosing $r_{n+1}$ sufficiently large, we can ensure that: (1) $\norm{\beta_{n+1}} < 2^{-n}\epsilon_n$;  
(2)
 \aryst
f_{n+1} &=& H_{n+1} T_{\omega_{n+1}} (H_{n+1})^{-1} \\
&=& H_{n}(h_{n+1} T_{\omega_{n}} T_{\beta_{n+1}} h_{n+1}^{-1}) H_{n}^{-1} \\
&=& H_{n}(T_{\omega_n} h_{n+1} T_{\beta_{n+1}} h_{n+1}^{-1}) H_{n}^{-1}
\earyst
is $2^{-n-1}\epsilon_{n}-$close to $f_n$ in $\Diff^{\infty}(\T^2) \cap \mathrm{Homeo}_*(\T^2)$; and (3), by a similar reason as above, $F_{n+1}$ is $2^{-n-1}\epsilon_{n}-$close to $F_n$ in $C^0(\R^2)$.
 This verifies $(a6)_{n}$.
Moreover, note that for any $m = kq_{n}$ with $k \in \Z$, we have $m \omega_{n+1} = k\hat{\omega}_{n} + \frac{(k, kv)}{r_{n+1}}$, and hence
\aryst
f_{n+1}^{m} = H_{n+1} T_{ \frac{(k, kv)}{r_{n+1}}} H_{n+1}^{-1}.
\earyst
By our choices of $\beta_{n+1}$, $v$, $r_{n+1}$ and $\kappa$ (see \eqref{PartI3kappa 2 n 1}), it is direct to see that:

(1) for any $\kappa$-dense subset of $\T^2$, denoted by $K$, the set $H_{n+1}(K)$ is $2^{-n-1}$-dense in $\T^2$;

(2) for any $z \in \T^2$,  $\{ (z + m\omega_{n+1}) \mod \Z^2 \}_{m \in \N }$ is $\kappa$-dense in $\T^2$.

Thus for any $z \in \T^2$, the set $\{ f_{n+1}^{m}(z)  \}_{m \in \N} = \{H_{n+1}(H_{n+1}^{-1}(z) + m \omega_{n+1})   \}_{m \in \N}$ is $2^{-n-1}$-dense in $\T^2$. This verifies $(a5)_{n+1}$. Moreover 
for some $m \in \N$,  $f_{n+1}^{m} = G_{n}^{\gamma}$
for some $\gamma$ with $\norm{\gamma} < \kappa$. 
Then by \eqref{PartI3d G G}, we verify the second part of  $(a3)_{n+1}$.

The above discussions show that, by choosing $r_{n+1}$ sufficiently large, we can ensure that  $(h_{n+1}$, $\omega_{n+1}, q_{n+1}, \hat{\omega}_{n+1})$  satisfies $(a1)_{n+1}-(a5)_{n+1}$ and $(a6)_n$, and thus complete the induction. Moreover, by choosing $r_{n+1}$ sufficiently large at each step of the induction, it is easy to ensure that the limit of $\omega_{n}$ satisfies \eqref{eq:super-liouvullen}.

We construct the sequence $\{ f_n \}_{n \geq 1}$ by induction.
By $(a6)$, $\{ f_n \}_{n \geq 1}$ converges in the $C^{\infty}$ topology to some map $f  \in \Diff^{\infty}(\T^2, \rm Vol) \cap Homeo_*(\T^2)$; $\{ \omega_n \}_{n \geq 1}$ converges to some $\omega \in \R^2$ satisfying \eqref{eq:super-liouvullen}; and $\{ F_n \}_{n \geq 1}$ converges to  some $F \in C^{0}(\R^2)$, which is clearly a lift of $f$. Moreover for any integer $n \geq 1$, we have $d_{\diff^{\infty}(\T^2)}(f_{n}, f), d_{C^0}(F_{n}, F), \norm{\omega - \omega_n} < \epsilon_n$. By \eqref{PartI3term 1}, $f$ has bounded mean motion. By \eqref{PartI3term 2}, for any $\epsilon > 0$,  there exist $x,y \in \T^2$ satisfying $d(x,y) < \epsilon$, and an integer $m > 0$, such that $d(f^{m}(x), f^{m}(y)) \geq \frac{1}{1000}$.  By \eqref{PartI3term 3}, we can see that  $\omega$ is totally irrational and super-Liouvillean \eqref{eq:super-liouvullen}.  By \eqref{PartI3term 4}, $f$ is minimal. This concludes the proof.
\end{proof}

\begin{proof}[Proof of Theorem \ref{thm example}]
For $d = 2$, Theorem \ref{thm example} is reduced to Proposition \ref{prop example}. 
Indeed, the last condition in Proposition \ref{prop example}  implies that $f$ is not topologically conjugate to a translation; and by Theorem \ref{thmJager}, $f$ is semi-conjugate to a translation by a surjection homotopic to $\id_{\T^2}$.
For $d > 2$, let $f'$ be given by Proposition \ref{prop example}, and let $\alpha \in \R^{d-2}$ be a totally irrational vector to be determined in due course. We set $f = f' \times T_{\alpha} \in \diff^{\infty}(\T^d, {\rm Vol})$. Observe that for any $\varepsilon > 0$, there exist $x,y \in \T^d$, and an integer $N  > 0$ such that $d(f^{N}(x), f^{N}(y)) \geq \frac{1}{1000}$. Thus $f$ is not topologically conjugate to a translation, but is semi-conjugate to one. Since $f'$ is minimal and $\T^2$ is connected, it is clear that for any integer $\ell > 0$, $f^{\ell}$ is also minimal. Then by letting $\alpha$ to be sufficiently well-approximated by rational vectors, we can ensure that $f$ is also minimal on $\T^d$, and $\rho(f)$ is super-Liouvillean. This concludes the proof.
\end{proof}

\section*{\bf{Acknowledgements}} 
We would like to thank Artur Avila, Bassam Fayad, Patrice Le Calvez and Rapha\"el Krikorian for their helpful conversations and comments. The first author would like to thank Yinshan Chang and Ruijun Wu for the beginning of the discussions. The second author wishes to thank Nikolaos Karaliolios and Sebastian van Strien for asking him the question on linearization, and remarks which helped improving the text; also Guan Huang and Jianlu Zhang for stimulating conversations.

%
%
%

\end{document}